\documentclass{article}

\usepackage[english]{babel}

\usepackage{amsmath}
\usepackage{graphicx}
\usepackage[colorlinks=true, allcolors=blue]{hyperref}

\usepackage{xcolor}

\newtheorem{theorem}{Theorem}
\newtheorem{lemma}{Lemma}
\newtheorem{definition}{Definition} 
\newtheorem{corollary}{Corollary}

\newtheorem{alg}{Algorithm}
\newtheorem{remark}{Remark} 
\newtheorem{proposition}{Proposition} 
\newtheorem{example}{Example}
\newenvironment{proof}{\paragraph{Proof:}}{\hfill$\square$}
\usepackage[colorinlistoftodos]{todonotes}
\usepackage[latin9]{inputenc}
\usepackage{graphicx}
\usepackage{babel}
\usepackage[a4paper]{geometry}

\usepackage[matrix,arrow,curve,cmtip]{xy}

\usepackage{txfonts}
\DeclareMathAlphabet{\mathcal}{OMS}{cmsy}{m}{n} 

\usepackage{
  amsmath,  
  amssymb,  
  tikz,     
  youngtab, 
  ytableau, 
  hyperref, 
  url       
}

\usepackage{amscd}
\usepackage{euscript}
\usepackage{latexsym}
\usepackage{keyval}
\usepackage{mathtools}
\usepackage{tikz}
\usetikzlibrary{arrows,shapes,snakes,automata,backgrounds,petri,through,positioning}
\usetikzlibrary{intersections}
\usepackage[symbol*]{footmisc}
\usepackage{verbatim}

\usepackage{mathrsfs}

\usepackage{BOONDOX-frak} 
\usepackage{graphicx}
\newtheorem{construction}[theorem]{Construction}

\usepackage{tikz}
\usetikzlibrary{intersections,through,backgrounds}
\usetikzlibrary{patterns}

    \pgfdeclarepatternformonly{south west lines}{\pgfqpoint{-0pt}{-0pt}}{\pgfqpoint{3pt}{3pt}}{\pgfqpoint{3pt}{3pt}}{
        \pgfsetlinewidth{0.4pt}
        \pgfpathmoveto{\pgfqpoint{0pt}{0pt}}
        \pgfpathlineto{\pgfqpoint{3pt}{3pt}}
        \pgfpathmoveto{\pgfqpoint{2.8pt}{-.2pt}}
        \pgfpathlineto{\pgfqpoint{3.2pt}{.2pt}}
        \pgfpathmoveto{\pgfqpoint{-.2pt}{2.8pt}}
        \pgfpathlineto{\pgfqpoint{.2pt}{3.2pt}}
        \pgfusepath{stroke}}
        
    \pgfdeclarepatternformonly{south east lines}{\pgfqpoint{-0pt}{-0pt}}{\pgfqpoint{3pt}{3pt}}{\pgfqpoint{3pt}{3pt}}{
    \pgfsetlinewidth{0.4pt}
    \pgfpathmoveto{\pgfqpoint{0pt}{3pt}}
    \pgfpathlineto{\pgfqpoint{3pt}{0pt}}
    \pgfpathmoveto{\pgfqpoint{.2pt}{-.2pt}}
    \pgfpathlineto{\pgfqpoint{-.2pt}{.2pt}}
    \pgfpathmoveto{\pgfqpoint{3.2pt}{2.8pt}}
    \pgfpathlineto{\pgfqpoint{2.8pt}{3.2pt}}
    \pgfusepath{stroke}}

\begin{document}

\markboth{Authors' Names}{Instructions for Typesetting Manuscripts using \LaTeX}


\title{Circle graphs (chord interlacement graphs) of Gauss diagrams: Descriptions of realizable Gauss diagrams, algorithms, enumeration
}

\author{Abdullah~Khan \\
University of Essex, UK \\ak20749@essex.ac.uk  
\and 
Alexei~Lisitsa\\
University of Liverpool, UK\\ a.lisitsa@liverpool.ac.uk  
\and  
Viktor~Lopatkin\\
HSE University, Moscow, Russia \\
wickktor@gmail.com
\and  
Alexei~Vernitski\\ 
University of Essex, UK\\asvern@essex.ac.uk}


%

\maketitle
\begin{abstract}
Chord diagrams, under the name of Gauss diagrams, are used in low-dimensional topology as an important tool for studying curves or knots. Those Gauss diagrams that correspond to curves or knots are called realizable. The theme of our paper is the fact that realizability of a Gauss diagram can be expressed via its circle graph. Accordingly, one can define and study realizable circle graphs (with realizability of a circle graph understood as realizability of any one of chord diagrams corresponding to the graph). Several studies contain theorems purporting to prove the fact. We check several of these descriptions experimentally and find counterexamples to the descriptions of realizable Gauss diagrams in some of these publications. We formulate new descriptions of realizable circle graphs and present an elegant algorithm for checking if a circle graph is realizable. We enumerate realizable circle graphs for small sizes and comment on these numbers. Then we concentrate on one type of curves, called meanders, and study the circle graphs of their Gauss diagrams.

\end{abstract}

{\small \;\;\;\; {\bf keywords:} Gauss diagrams, realiziability criteria, circle graphs, interlacement graphs}

\section{Introduction}

Consider a closed planar curve which, possibly, crosses itself at some points; see Figure \ref{fig:gauss-ex}(a) for an example. Build a graph in which crossings are vertices, and there is [there isn't] an edge between two crossings $c$, $d$ if respectively it isn't [it is] possible to travel along the curve from $c$ to $c$ without passing through $d$; the graph corresponding to the planar curve in Figure \ref{fig:gauss-ex}(a) is shown in Figure \ref{fig:gauss-ex}(c). If a graph can be produced in this way on the basis of some closed planar curve, we will say that the graph is \emph{realizable}. 
The standard objects which are widely studied and which are related to realizable graphs are realizable Gauss words and realizable Gauss diagrams; let us define these concepts. Consider again a closed planar curve. Choose a point on the curve to start from, and then travel along the curve all the way until you come back to where you started, recording the crossings you encounter. For example, if you start at the top of the curve in Figure \ref{fig:gauss-ex}(a) and travel along the curve to the right initially, you encounter the crossings in the order $34124123$; this is a word corresponding to this curve. Obviously, there are many words corresponding to a curve, depending on the point where you start and the direction in which you travel along the curve. 

It is easy to notice that a word corresponding to a curve contains each letter exactly twice. This is why a \emph{Gauss word} is defined as a word which contains each letter exactly twice. A Gauss word is called \emph{realizable} if it corresponds to some planar closed curve.

A \emph{chord diagram} is a convenient visual representation of a Gauss word; it consists of a Gauss word written around a circle, with chords inside the circle connecting the two occurrences of each of the letters. For example, the chord diagram of the Gauss word $34124123$ is shown in Figure \ref{fig:gauss-ex}(b). It is easy to see that, conveniently, all Gauss words corresponding to the same curve share the same chord diagram. A chord diagram is called \emph{realizable} if it corresponds to some planar closed curve. 

A \emph{Gauss diagram} is defined in literature either as a chord diagram or as a chord diagram having one additional property (the first parity condition, introduced below in Proposition \ref{prop:first-parity}). It should not cause confusion, but where it might, we will state explicitly whether the first parity condition is assumed or not.

One can define a graph corresponding to a chord diagram; this graph is called a \emph{circle graph} by graph theorists or a \emph{chord interlacement graph} by knot theorists. Build a graph in which chords are vertices, and there is an edge between two chords $c$, $d$ if $c$ and $d$ intersect in the chord diagram. For example, the graph shown in Figure \ref{fig:gauss-ex}(c) happens to be the graph corresponding to the Gauss diagram in Figure \ref{fig:gauss-ex}(b). In general, it is easy to see that the graph corresponding to a curve coincides with the graph corresponding to the Gauss diagram of the curve.

An important question which immediately arises is whether the concepts of realizability of graphs and of Gauss diagrams are consistent with each other. The answer is positive, as the following statement shows.

\begin{theorem}
 A graph is realizable if and only if it corresponds to a realizable Gauss diagram. If a graph corresponds to several Gauss diagrams then either all these Gauss diagrams are realizable or none is.
\end{theorem}
\begin{proof}
 The statement follows from the fact that there are descriptions of realizability of Gauss diagrams \cite{rosenstiehl:hal-00259712,10.1007/3-540-63938-1_65,DBLP:journals/dm/ShtyllaTZ09} which can be worded entirely in terms of the corresponding graphs, as discussed in more detail in Section \ref{sec:realisability-descriptions}.
\end{proof}

One necessary condition of realizability, which we call the \emph{first parity condition}, is the following. 

\begin{proposition}[First parity condition] \label{prop:first-parity}
 1) In a realizable graph the degree of each vertex is even.
 \\2) In a realizable Gauss word the distance between the positions of the two occurrences of each letter is odd.
 \\3) In a realizable chord diagram, for each chord, the number of chords intersecting this chord is even.
\end{proposition}

The part (3) of Proposition \ref{prop:first-parity} is due to Gauss; this is why Gauss words and Gauss diagrams are called after him. Gauss also posed the problem \cite{Gauss} asking when a Gauss diagram is realizable. The problem was solved for the first time by Dehn in 1936 \cite{dehn1936}. Since then many efficient criteria and algorithms for checking the realizability of Gauss diagrams 
have been developed
\cite{10.2307/2037443,FRANCIS1969331, DEFRAYSSEIX198129,lovasz1976,rosenstiehl:hal-00259712, rosenstiehl:hal-00259721,DBLP:journals/jal/RosentiehlT84,10.1007/3-540-63938-1_65,DOWKER198319,vena2018topological,ce-pp2-96,cw-prpmg-94}. 
Talking about computational complexity, the realizability of Gauss diagrams or words can be checked in polynomial time 
and even in linear time \cite{DBLP:journals/jal/RosentiehlT84}. Furthermore, as it is shown in \cite{10.1007/978-3-642-21254-3_29}  it can be done with logarithmic space complexity. 

Most of the research on Gauss realiziability has focused on the  important class of chord diagrams called \emph{prime} diagrams, which  can be defined as follows.   One says that a closed planar curve $C$ is a connected sum of two closed planar curves $C_1, C_2$ if $C$ can be produced from $C_1$ and $C_2$ by `cutting' $C_1$ in one place, `cutting' $C_2$ in one place, and then `gluing' the free ends of $C_1$ to the free ends of $C_2$ without introducing (or erasing) any crossings. See an example in Figure \ref{fig:connected-sum}. It is easy to see that a curve can be decomposed into a connected sum if an only if the corresponding graph is disconnected, see Figure \ref{fig:gauss-ex}(c). It is also easy to see that the graph corresponding to a Gauss diagram is disconnected if and only if the set of chords of the Gauss diagram can be decomposed into families so that none of the chords in one family intersect the chords of another family; compare Figure \ref{fig:gauss-ex}(b) and Figure \ref{fig:gauss-ex}(c). A Gauss diagram whose corresponding graph is connected (disconnected) is called prime (composite), respectively.

Denote the one dimensional circle by $S^{1}$ and the two-dimensional euclidean plane by $R^{2}$. Two closed planar curves $\gamma_{1}, \gamma_{2} : S^{1} \rightarrow R^{2}$ are called 
\emph{equivalent}  if there is a homeomorphism 
$h: R^{2} \rightarrow R^{2}$ such that $h\gamma_{1} = \gamma_{2}$. Two Gauss words over the same alphabet are  \emph{isomorphic}  if one can be obtained from the other by cyclic shifts and reversing the word \cite{c-cic-91}. Similarly, the isomorphism of Gauss diagrams can be defined. In general, there are `more' classes of equivalent closed planar curves than classes of isomorphic Gauss diagrams (see examples in \cite{Valette16}); however, there is a one-to-one correspondence when one considers planar curves with prime Gauss diagrams. Indeed, for planar curves $\gamma_{1}$ and $\gamma_{2}$, such that their Gauss diagrams are prime,  both  corresponding  Gauss words and Gauss diagrams are isomorphic if and only if $\gamma_{1}$ and $\gamma_{2}$ are equivalent,  see e.g.  \cite{Soulie2004,chmutov-2006}.

The main goal of the research presented in this paper is the  experimental investigation of several realizability descriptions expressible in terms of the circle interlacement graphs, enumeration of classes of non-equivalent chord diagrams and corresponding graphs up to isomorphisms.

To this end we have implemented two algorithms for chord diagrams generation: one is a \emph{direct} algorithm based on a bijection between odd-even matchings of the set $[0, \ldots 2n-1]$ and   permutations on $n$ elements\cite{khan2021experimental,KLV21-lintel}; another one is a novel and more efficient \emph{incremental} algorithm which builds all chord diagrams of the size $n+1$ based on the  set of such diagrams of the size $n$. 
Using these algorithms, we  enumerated the classes of non-equivalent chord diagrams of sizes up to n=13 and tested different  realizability descriptions expressed in terms of interlacement graphs. We have confirmed experimentally the validity of some of these descriptions, especially STZ from \cite{DBLP:journals/dm/ShtyllaTZ09}; at the same time, we have discovered that recently proposed simple GL \cite{grinblat2018realizability,doi:10.1142/S0218216520500315} and B \cite{doi:10.1142/S0218216519500159} conditions     are wrong. We have found the minimal counterexample of size n=9 for both these descriptions (first reported in \cite{khan2021experimental}), and also enumerated counterexamples for n=10 and 11.  We reflect on the counterexamples and highlight the error in the arguments of \cite{grinblat2018realizability,doi:10.1142/S0218216520500315} and \cite{doi:10.1142/S0218216519500159}. We provide a correction and propose  new complete realizability  criteria. 
Then we enumerate non-isomorphic circle interlacement graphs for all realizable prime chord diagrams up to the size n = 13. We cross-validate and explain these numbers by comparison with existing empirical data on mutant knots. We also enumerate all non-isomorprhic graphs of a special and natural class of  meander diagrams.  We cross-validate our enumeration results with the OEIS repository and other published works. In many cases we have produced new results, and two of our new sequences of numbers have been accepted to the OEIS as A343358 and A338660.

\begin{figure}[h!]
    \centering
    \begin{tikzpicture}[scale=0.8]

\draw[line width =2, name path= a](0,0) to [out= 30, in = 30] (-0.5, 2.5);
\draw[line width =2, name path= b](-0.5,2.5) to [out= 210, in = 180] (1.2, -1);
\draw[line width =2, name path= c] (1.2,-1) to [out= 0, in = 0] (-0.5, 3.3);
\draw[line width =2, name path= d] (-0.5,3.3) to [out= 180, in = 160] (-2, -1);
\draw[line width =2, name path= e] (-2,-1) to [out= 340, in = 270] (1.2, 0.5);
\draw[line width =2, name path= f] (1.2,0.5) to [out= 90, in = 0] (-1.2, 1.5);
\draw[line width =2, name path= g] (-1.2,1.5) to [out = 180, in = 90] (-2, 0.75) to [out= 270, in = 210] (0, 0);

 \fill [name intersections={of=f and b, by={1}}]
(1) circle (3pt) node[above left] {$1$};

 \fill [name intersections={of=a and f, by={2}}]
(2) circle (3pt) node[above right] {$2$};

 \fill [name intersections={of=e and b, by={3}}]
(3) circle (3pt) node[above] {$3$};

 \fill [name intersections={of=g and b, by={4}}]
(4) circle (3pt) node[below] {$4$};
\begin{scope}[xshift = 5cm, yshift = 1cm]
     \draw[line width =2] (0,0) circle (2.2);
     
     {\foreach \angle/ \label in
   { 90/1, 135/4, 180/2, 225/1, 270/4, 315/3, 0/3, 45/2}
   {
    \fill(\angle:2.5) node{$\label$};
    \fill(\angle:2.2) circle (3pt) ;
    }
}
     \draw[line width = 2] (90:2.2) -- (225:2.2);
     \draw[line width = 2] (135:2.2) -- (270:2.2);
     \draw[line width = 2] (180:2.2) -- (45:2.2);
     \draw[line width = 2] (0:2.2) to [out = 180, in = 140] (315:2.2);
     
\end{scope}
\begin{scope}[xshift =10cm, yshift = 0.5cm]
  \draw[line width = 2] (90:1.7) -- (200:1.7);
  \draw[line width = 2] (90:1.7) -- (340:1.7);
  \draw[line width = 2] (200:1.7) -- (340:1.7);
  
  \fill(90:2) node{$1$};
  \fill(90:1.7) circle(3pt);
  
  \fill(200:2) node{$2$};
  \fill(200:1.7) circle(3pt);
  
  \fill(340:2) node{$4$};
  \fill(340:1.7) circle(3pt);
  
  \fill(40:3) node{$3$};
  \fill(40:2.7) circle(3pt);
  
 \end{scope}
  \fill(0,-2.5) node{$a)$};
  \fill(5,-2.5) node{$b)$};
  \fill(10,-2.5) node{$c)$};
    \end{tikzpicture}
    \caption{Example of a) a planar curve; b) its Gauss diagram and c) its interlacement graph. The corresponding Gauss word is $\mathbf{12334124}$}
    \label{fig:gauss-ex}
\end{figure}
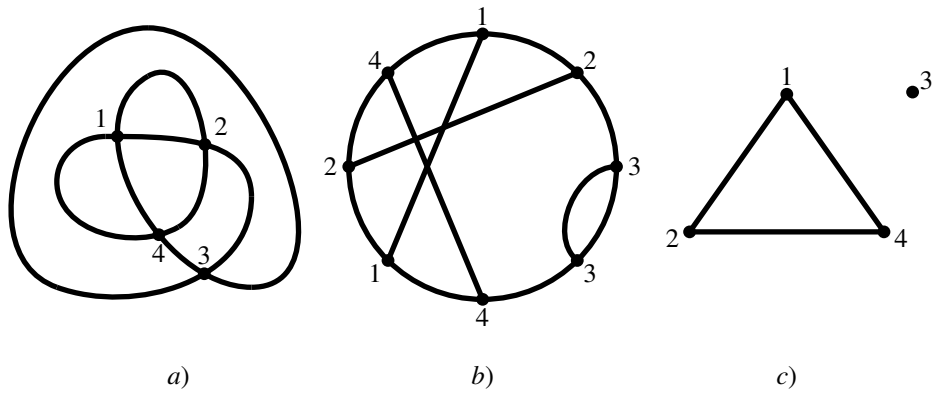


\begin{figure}[h!]
\centering
\begin{tikzpicture}[scale =0.9]
\begin{scope}[xshift = -2cm]
 \draw[line width =2, name path= a](0,0) to [out= 30, in = 30] (-0.5, 2.5);
 \draw[line width =2, name path= b](-0.5,2.5) to [out= 210, in = 180] (1.2, -1);
 \draw[line width =2, name path= c] (1.2,-1) to [out= 0, in = 0] (-0.5, 3.3);
 \draw[line width =2, name path= d] (-0.5,3.3) to [out= 180, in = 160] (-2, -1);
 \draw[line width =2, name path= e] (-2,-1) to [out= 340, in = 270] (1.2, 0.5);
 \draw[line width =2, name path= f] (1.2,0.5) to [out= 90, in = 0] (-1.2, 1.5);
 \draw[line width =2, name path= g] (-1.2,1.5) to [out = 180, in = 90] (-2, 0.75) to [out= 270, in = 210] (0, 0);

 \fill [name intersections={of=f and b, by={1}}]
(1) circle (3pt) node[above left] {$1$};

 \fill [name intersections={of=a and f, by={2}}]
(2) circle (3pt) node[above right] {$2$};

 \fill [name intersections={of=e and b, by={3}}]
(3) circle (3pt) node[below] {$3$};

 \fill [name intersections={of=g and b, by={4}}]
(4) circle (3pt) node[below] {$4$}; 
\end{scope}

\begin{scope}[xshift = 4cm]
 \draw[line width =2, name path= a](0,0) to [out= 30, in = 30] (-1, 2.5);
 \draw[line width =2, name path= b](-1,2.5) to [out= 210, in = 250] (0.8, -0.5);
 \draw[line width =2, name path= f] (0.8,-0.5) to [out= 70, in = 0] (-1.4, 1.5);
 \draw[line width =2, name path= g] (-1.4, 1.5) to [out = 180, in = 90] (-2, 0.75) to [out= 270, in = 210] (0, 0);
 
 \fill [name intersections={of=f and b, by={1}}]
(1) circle (3pt) node[above left] {$1$};

 \fill [name intersections={of=a and f, by={2}}]
(2) circle (3pt) node[above right] {$2$};

  \fill [name intersections={of=g and b, by={4}}]
(4) circle (3pt) node[below] {$4$}; 
 \end{scope}

\begin{scope}[xshift = 8cm]
 \draw[line width =2, name path= a](0,0) to [out= 300, in = 270] (2, 1);
 \draw[line width =2](2,1) to [out= 90, in = 0] (0, 3);
 \draw[line width =2](0,3) to [out= 180, in = 150] (-1, -1);
 \draw[line width =2, name path= d](-1,-1) to [out= 330, in = 300] (1, 1);
 \draw[line width =2](1,1) to [out = 120, in = 0] (0.4,1.3) to [out= 180, in = 120] (0, 0);
 
   \fill [name intersections={of=d and a, by={3}}]
(3) circle (3pt) node[below] {$3$}; 
\end{scope}

  \fill(1,0.8) node{{\Huge{=}}};
  \fill(5.5,0.8) node{{\Huge{\#}}};
  \draw[line width = 1, dotted] (-2.5,-0.7) -- (-0.5,0);

\end{tikzpicture}
\caption{The curve on the left is a connected sum of the other two.}
\label{fig:connected-sum}
\end{figure}
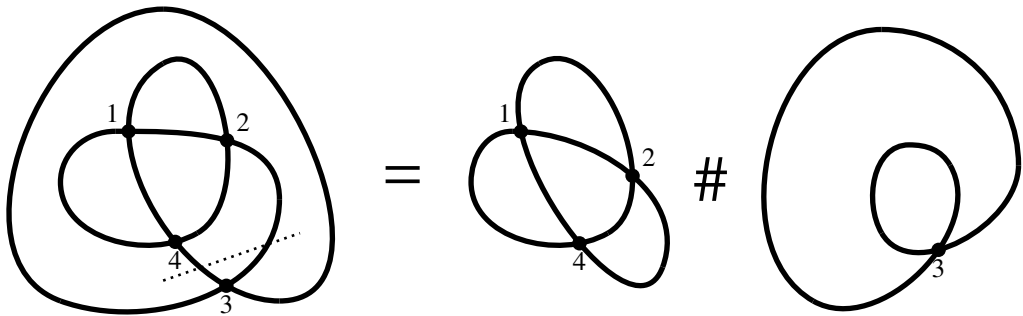



\section{Realizability descriptions based on interlacement graphs} \label{sec:realisability-descriptions}

In this section we survey known realiziability description for Gauss diagrams which can be expressed in terms of the interlacement graphs.

\subsection{R-conditions} 

For the first time the
relizability description for Gauss words and diagrams expressed solely in terms of interlacement graphs was 
found by Rozenstiehl in 
\cite{rosenstiehl:hal-00259712}. Following  \cite{10.1007/3-540-63938-1_65} Rozentiehl's conditions can be formulated as follows. 
A Gauss diagram $g$ is realizable iff its interlacement graph $I_{g}$ has the following properties: 

\begin{itemize}
    \item the first parity condition is satisfied, that is, the degree of each vertex is even; \hspace*{8mm} {\bf (PC1)}
    \item there is a subset of vertices $A \subset V$ of $I_{g}$ such that the following two conditions are equivalent for any two vertices $u$ and $v$:  
    \begin{itemize}
    \item the vertices u and v have an odd number of common neighbours, 
    \item the vertices u and v are neighbours and either both are in $A$ or neither is in $A$  
    \end{itemize}
\end{itemize}

\subsection{STZ-conditions} 

Shtylla, Traldi and Zulli 
in \cite{DBLP:journals/dm/ShtyllaTZ09}, while keeping {\bf PC1} condition,  
presented an algebraic re-formulation of the second Rosentiehl's condition: 

\begin{itemize}
\item For its adjacency matrix $M_{g}$ there is a 
    diagonal matrix $\Lambda_{g}$ such that $M_{g}+\Lambda_{g}$ is an idempotent matrix (all matrices are considered over $GF(2)$, the finite field of two elements). 
\end{itemize}

\subsection{GL-conditions}
Another desciption for realizability expressible in terms of interlacement graph  was proposed by Grinblat and Lopatkin in  \cite{grinblat2018realizability,doi:10.1142/S0218216520500315}. It is claimed that a prime Gauss diagram $g$ is realizable if and only if the following conditions hold true: 

\begin{itemize}
    \item In its interlacement graph $I_{g}$ each pair of non-neighbouring vertices has an \emph{even} number of common neighbours (possibly, zero).  $\hspace*{9cm}$   ({\bf PC2})
    \item The above condition {\bf PC2} holds true for the reduced graph $I_{g}/v$ for each vertex $v$ of $I_{g}$.   
\end{itemize}

For a vertex $v$ of $I_{g} = \langle V,E\rangle$ the reduced graph $I_{g}/v = \langle V',E' \rangle$ is defined as follows.
$V' = V -\{v\}$ and  $E' = \{ (v_{1},v_{2}) \in E |  ((v,v_{1}) \not\in E)\lor (v,v_{2}) \not\in E))  \} \cup  \{ (v_{1},v_{2}) \mid (v,v_{1}) \in E \;\& \; (v,v_{2}) \in E \;\&\; (v_{1},v_{2}) \not\in E\}$. See also further discussion of this operation in terms of intersecting chords in Section 2 of   
\cite{doi:10.1142/S0218216519500159}.

\subsection{B-conditions} 

Biryukov in \cite{doi:10.1142/S0218216519500159} proposes a realizability description which is even simpler 
than GL-conditions. 
It is claimed that a prime Gauss diagram $g$ is realizable  if and only if 
the following conditions for 
$I_{g} = \langle V,E\rangle$ 
hold true: 

\begin{itemize}
    \item It  satisfies both above {\bf PC1} and {\bf PC2} conditions, which is called a \emph{strong parity condition} in  \cite{doi:10.1142/S0218216519500159};
    \item  For any three pairwise connected vertices  $a, b, c \in V$   the sum of the number of vertices  adjacent to  $a$, but not adjacent to  $b$ nor $c$, and the number of vertices adjacent to  b and c, but not adjacent to  $a$, is even.   \hspace*{12cm}       {\bf(PC3)}
\end{itemize}

\begin{remark}
 The B-conditions have already appeared in \cite{grinblat2018realizability} and been presented in \cite{grinblat2018realizability,doi:10.1142/S0218216520500315} as necessary conditions of realizability of Gauss diagrams in terms of adjacency matrix. These conditions appeared in the proof of \cite[Theorem 4.2]{grinblat2018realizability} (see also the proof of \cite[Theorem 4.3]{doi:10.1142/S0218216520500315}) as follows $|A| + |B| \equiv  \bmod (2)$ where $A$ and $B$ are exactly sets mentioned in \textbf{PC3} conditions.  
\end{remark}

\section{Algorithms for producing Gauss diagrams}

In this section we describe two  algorithms which we devised and used 
for the generation of Gauss diagrams.

\subsection{Even-odd Matchings and Permutation-based algorithm}

Let  $V_{2n}= \{0, \ldots, 2n-1\}$. A family of subsets $S_{j} \subseteq V_{2n}$, $j = 1, \ldots n$  is called an even-odd matching of $V_{2n}$ iff 1) $\cup  S_{j} = V_{2n}$; 2) $S_{i} \cap S_{j} = \emptyset$, $i \not= j$; 3) each $S_{j}$ contains exactly one even and one odd number.   Even-odd matchings serve as natural encodings of Gauss diagrams; indeed, place the numbers $\{0,\ldots, 2n-1\}$ clockwise around a circle and connect those in each $S_{j}$ by a chord. Obviously, the first parity condition {\bf PC1} will be satisfied in the produced chord diagram.    

For efficient Gauss diagram  generation we use the the following simple bijective map from the symmetric group $S_{n}$ to the set $P_{2n}$ of all even-odd matchings on $V_{2n}= \{0, \ldots, 2n-1\}$.  
For  $\sigma \in S_{n}$,  $\sigma:[1..n]\rightarrow[1..n]$ , $\beta(\sigma) =  \{\{2*i-1,2*\sigma(i)-2\}| i \in [1..n]\}$. 

To deal with even-odd matchings algorithmically we use their list (or $n$-tuple)   representation $((a^1_1, a^2_1), \dots,$ $(a^1_n, a^2_n))$, called a \emph{lintel} in \cite{khan2021experimental}.   
Two lintels are \emph{equivalent} if they can be transformed into one another by steps of the following two types: 1) swapping the positions of two numbers in a chord; 2) swapping the positions of chords in the list. 3) shifting the value of each entry of the lintel cyclically modulo $2n$; 4) inverting the value of each entry of the lintel modulo $2n$. The equivalence of lintels corresponds to equivalence of Gauss diagrams.   
A lintel is a \emph{sorted lintel} if each pair in it is sorted, and first elements of pairs are sorted, that is, for each $i$ we have $a^1_i < a^2_i$, and for all $i < j$ we have $a^1_i < a^1_j$.
 In each class of equivalent lintels the sorted lintel which is the first in the lexicographic order of lintels 
(we call such lintels \emph{Lyndon lintels}, by analogy with Lyndon words~\cite{Lyndon})
can 
serve as the canonical representative of the class, and there is a one-to-one correspondence between Lyndon lintels and Gauss diagrams. 

Thus, we can use the following permutation-based algorithm for generating all non-equivalent Gauss diagrams satisfying some given properties $P$.  

\begin{alg}
For each permutation $\sigma \in S_{n}$:\\
1. Produce a lintel $l_{\sigma}$ corresponding to $\beta(\sigma)$. \\
2. Produce the canonical form (the Lyndon lintel) $l$   of $l_{\sigma}$\\
3. If $l$ satisfies $P$  and is not stored yet in the list $L$ then store $l$ in $L$.  

\end{alg}

\subsection{An algorithm for producing prime Gauss diagrams}

Our experiments to check the correctness of descriptions of realizability of
Gauss diagrams required us to generate all realizable and unrealizable Gauss diagrams of small sizes. From the form of the descriptions of realizable Gauss diagrams, we knew that if there is a counterexample, it is possible to find a counterexample which is a prime Gauss diagram. Therefore, for our experiments
it is sufficient to generate only all prime Gauss diagrams. 

Thus, we were motivated to look for an algorithm to generate efficiently all prime Gauss diagrams of a given size. We have found such an algorithm and used it in some of our experiments. This section describes this algorithm and proves that it works correctly. We assume that condition \textbf{PC1} is a part of the definition of a Gauss diagram, so you will see below that we set up our algorithm in such a way that it only produces all prime Gauss diagrams which satisfy condition \textbf{PC1}.

Recall that a Gauss diagram of size $n$ is a circle with $n$ chords; we refer to the endpoints of chords as \emph{vertices}. We call each part of the circumference of the circle between two consecutive vertices an \emph{arc}. Let $G$ be a prime Gauss diagram of size $n\ge2$. Consider an arc $(u,v)$ connecting vertices $u,v$. Since $G$ is prime, it contains no chord connecting $u$ and $v$ (indeed, if $G$ contained such a chord, this chord would be disconnected from the rest of the diagram). Therefore, $u$ and $v$ belong to two distinct chords $c_{u}$ and $c_{v}$. If the chords $c_{u}$ and $c_{v}$ interlace (don't interlace), we will say that the arc $(u,v)$ is a $\times$-arc (a $\parallel$-arc).

By a \emph{wheel Gauss diagram} we shall mean a Gauss diagram of size
$n$ whose list of vertices (listed consecutively around the circle)
are $v_{1},v_{2},\dots,v_{2n}$ and whose chords are $(v_{i},v_{n+i})$
for all $i$; informally speaking, each vertex is connected to the
opposite vertex lying on the opposite side of the circle\footnote{In graph-theoretical terms, the graph corresponding to this diagram
is not what is called a wheel graph, but what is called a Mobius ladder;
however, it would be a mouthful to talk about a `Mobius ladder Gauss
diagram', so we call them wheel diagrams.}. The following result is a simple observation.

\begin{lemma}
\begin{enumerate}
    \item 
 If $G$ is a prime Gauss diagram of size $n\ge2$ then some
of its arcs are $\times$-arcs; in other words, not all its arcs are
$\parallel$-arcs.

\item
Let $G$ be a Gauss diagram. All arcs of $G$ are $\times$-arcs
if and only if $G$ is a wheel diagram.
\end{enumerate}
\end{lemma} 

Consider a Gauss diagram $G$. Choose one of its chords $(u,v)$ and
one of its arcs $(p,q)$. Insert two new vertices $p^{\sharp},q^{\flat}$
between $p$ and $q$; in other words, split the arc $(p,q)$ into
$3$ consecutive arcs $(p,p^{\sharp})$, $(p^{\sharp},q^{\flat})$,
$(q^{\flat},q)$. Delete the chord $(u,v)$. Add either two new chords
$(u,p^{\sharp})$ and $(v,q^{\flat})$ or two new chords $(v,p^{\sharp})$
and $(u,q^{\flat})$; the choice between the former or the latter
pair of new chords is determined by the fact that exactly one of them
preserves condiction \textbf{PC1} in the newly formed diagram, so we choose
the one that does. Let us refer to the process we have just described as the
\emph{teepee move}, and to its inverse as the \emph{reverse teepee
move}. If the arc $(p^{\sharp},q^{\flat})$ in the diagram formed
by a teepee move is a $\times$-arc (a $\parallel$-arc), we will say
that it is a $\times$-move (a $\parallel$-move). Figure \ref{fig:A-teepee-move}
shows an example of a teepee move (a $\parallel$-move), by which
one chord (shown as a dashed line) is replaced by two new chords connected
to two new vertices.


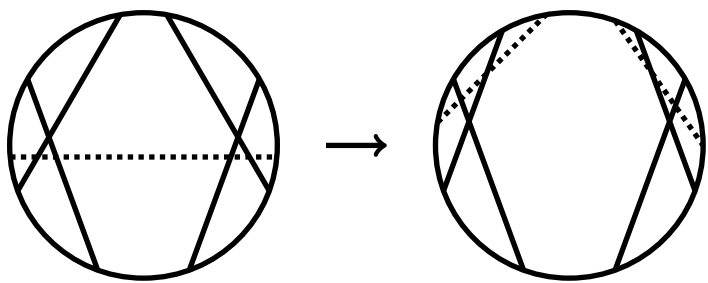
\begin{figure}[h!]
    \centering
     \begin{tikzpicture}[scale = 0.8]
      \begin{scope}
        \draw[line width =2] (0,0) circle (2.2);
        
        \draw[line width =2] (200:2.2) -- (100:2.2);
        \draw[line width =2] (150:2.2) -- (250:2.2);
        
        \draw[line width =2] (80:2.2) -- (340:2.2);
        \draw[line width =2] (30:2.2) -- (290:2.2);
        
        \draw[line width =2, dotted] (185:2.2) -- (355:2.2);
      \end{scope}
       \draw[line width =2,->] (0:3) -- (0:4);
      \begin{scope}[xshift =7cm]
       \draw[line width =2] (0,0) circle (2.2);
       
        \draw[line width =2] (200:2.2) -- (120:2.2);
        \draw[line width =2] (150:2.2) -- (250:2.2);
        
        \draw[line width =2] (60:2.2) -- (340:2.2);
        \draw[line width =2] (30:2.2) -- (290:2.2);
        
        \draw[line width =2, dotted] (170:2.2) -- (100:2.2);
        \draw[line width =2, dotted] (70:2.2) -- (0:2.2);
       
      \end{scope}
    \end{tikzpicture}
    \caption{A teepee move}
    \label{fig:A-teepee-move}
\end{figure}


\begin{theorem} \label{thm:teepee}
Suppose $G$ is a prime Gauss diagram of size $n\ge2$ which
is not a wheel. Then it can be produced by a teepee move from a prime
Gauss diagram of size $n-1$.
\end{theorem}

\begin{proof}
Since $G$ is not a wheel, by the lemma, it contains both $\times$-arcs
and $\parallel$-arcs. Consider two consecutive arcs such that one
of them is a $\times$-arc and the other is a $\parallel$-arc. Specifically,
suppose $(u,v)$ is a $\times$-arc and $(v,w)$ is a $\parallel$-arc.
Let us denote the chords connected to $u,v,w$ by $(u,u')$, $(v,v')$
and $(w,w')$. There are two possible configurations of these $3$ chords
within $G$, and they are presented in Figure \ref{fig:Chords-in-the}.
In the former case, perform a reverse teepee move removing vertices
$u$ and $v$ and connecting vertices $u'$ and $v'$ with a chord.
In the latter case, perform a reverse teepee move removing vertices
$v$ and $w$ and connecting vertices $v'$ and $w'$ with a chord.
In either case, by considering all possible positions of those chords
of $G$ which are not shown in the figure, one can see that after
this reverse teepee move the diagram remains prime. 
\end{proof}


\begin{figure}[h!]
    \centering
     \begin{tikzpicture}[scale =0.8]
     \begin{scope}
      \draw[line width =2] (0,0) circle (2.2);
       
       \draw[line width = 2] (80:2.2) -- (290:2.2);
       \draw[line width = 2] (110:2.2) -- (220:2.2);
       \draw[line width = 2] (140:2.2) -- (330:2.2);

       \fill(80:2.5) node{$w$};
       \fill(80:2.2) circle (3pt);
       
       \fill(110:2.5) node{$v$};
       \fill(110:2.2) circle (3pt);
       
       \fill(140:2.5) node{$u$};
       \fill(140:2.2) circle (3pt);
       
       \fill(220:2.5) node{$v'$};
       \fill(220:2.2) circle (3pt);
       
       \fill(290:2.5) node{$w'$};
       \fill(290:2.2) circle (3pt);
       
       \fill(330:2.5) node{$u'$};
       \fill(330:2.2) circle (3pt);
     \end{scope}

     \begin{scope}[xshift = 7cm]
      \draw[line width =2] (0,0) circle (2.2);
       
       \draw[line width = 2] (80:2.2) -- (310:2.2);
       \draw[line width = 2] (110:2.2) -- (220:2.2);
       \draw[line width = 2] (140:2.2) -- (270:2.2);

       \fill(80:2.5) node{$w$};
       \fill(80:2.2) circle (3pt);
       
       \fill(110:2.5) node{$v$};
       \fill(110:2.2) circle (3pt);
       
       \fill(140:2.5) node{$u$};
       \fill(140:2.2) circle (3pt);
       
       \fill(220:2.5) node{$v'$};
       \fill(220:2.2) circle (3pt);
       
       \fill(310:2.5) node{$w'$};
       \fill(310:2.2) circle (3pt);
       
       \fill(270:2.5) node{$u'$};
       \fill(270:2.2) circle (3pt);
     \end{scope}

     \end{tikzpicture}
    \caption{Chords in the proof of Theorem \ref{thm:teepee}}
    \label{fig:Chords-in-the}
\end{figure}
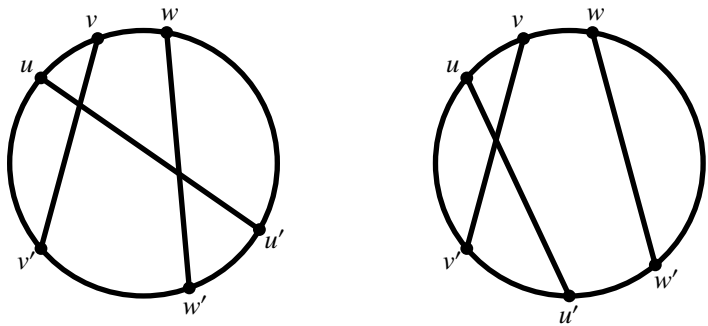


One might wonder whether the converse result is also true; indeed,
we saw that \emph{all} prime Gauss diagrams can be produced from smaller
prime Gauss diagrams using teepee moves; perhaps we also can claim
that \emph{only} prime Gauss diagrams can be produced from smaller
prime Gauss diagrams using teepee moves? Unfortunately, this is not
true. A counterexample is shown in Figure \ref{fig:A-teepee-move};
here a $\parallel$-move results in producing a composite Gauss diagram. 

In the proof of the theorem we explicitly employed both $\times$-moves
and $\parallel$-moves. One might wonder if it is sufficient to use
either only $\times$-moves or only $\parallel$-moves to produce
all prime Gauss diagrams. The answer is negative. In Figure \ref{fig:Producing-Gauss-diagrams}
we present an example of a prime Gauss diagram which cannot be produced
by a $\times$-move from a prime Gauss diagram and an example of a
prime Gauss diagram which cannot be produced by a $\parallel$-move
from a prime Gauss diagram. 



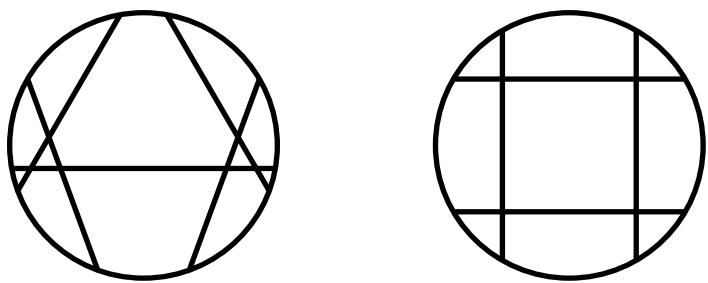
\begin{figure}[h!]
    \centering
     \begin{tikzpicture}[scale=0.8]
      \begin{scope}
        \draw[line width =2] (0,0) circle (2.2);
        
        \draw[line width =2] (200:2.2) -- (100:2.2);
        \draw[line width =2] (150:2.2) -- (250:2.2);
        
        \draw[line width =2] (80:2.2) -- (340:2.2);
        \draw[line width =2] (30:2.2) -- (290:2.2);
        
        \draw[line width =2] (190:2.2) -- (350:2.2);
      \end{scope}

      \begin{scope}[xshift =7cm]
        \draw[line width =2] (0,0) circle (2.2);
       
        \draw[line width =2] (120:2.2) -- (240:2.2);
        \draw[line width =2] (60:2.2) -- (300:2.2);
        
        \draw[line width =2] (30:2.2) -- (150:2.2);
        \draw[line width =2] (330:2.2) -- (210:2.2);
        
      \end{scope}
    \end{tikzpicture}
    \caption{Producting Gauss diagrams by $\times$-moves and $\parallel$-moves}
    \label{fig:Producing-Gauss-diagrams}
\end{figure}


Thus, we have designed the following useful algorithm for producing
all prime Gauss diagrams of a given size.
\begin{enumerate}
\item Start with a list of all prime Gauss diagrams of size $n-1$.
\item Apply to each of these Gauss diagrams the teepee move in all possible
combinations of the diagram's chords and arcs.
\item Some of Gauss diagrams produced by the teepee moves will be composite;
discard them.
\item Some Gauss diagrams will be produced multiple times; discard the duplicates.
\item Add the wheel diagram of size $n$.
\item What we have produced is a list of all prime Gauss diagrams of size
$n$. 
\end{enumerate}
\subsection{Implementation details} 
We have implemented both permutation-based and incremental  algorithms, as well as  classical  algorithm (CA) for realiziability checking from \cite{dehn1936,KAUFFMAN1999663},        
in logic programming language SWI-Prolog~\cite{wielemaker:2011:tplp}. 

Efficient built-in predicates \verb"Permutation/2" and \verb"sort/4" are used to  generate the permutations and to sort the lintels. The dynamic Prolog facts are used to  store intermediate and final results.

\section{Diagrams generation experiments and  Counterexamples} \label{sec:counterexamples}

Using our implementation, we have conducted a range of experiments and enumerated the classes of non-equivalent Gauss diagrams satisfying various    
combinations of properties.  

\begin{table}[h]
 \centering
    \begin{tabular}{|l|c|c|c|c|c|c|c|c|c|c|c|} 
   \hline 
   \it{   \;\;\;\;\;\; Size}  & 3 & 4 & 5 & 6 & 7 & 8 & 9 & 10 & 11 & 12
     \\
      \hline
      Realisability (CA) & 1 & 1 & 2 & 3 & 10 & 27 & 101 & 364 & 1610 & 7202 
      \\
      STZ & 1 & 1 & 2 & 3 & 10 & 27 & 101 & 364 & 1610 & 7202
      \\
      B  & 1 & 1 & 2 & 3 & 10 & 27 & {\bf 102} & {\bf 370} & {\bf 1646} & {\bf 7437}
      \\
      GL & 1 & 1 & 2 & 3 & 10 & 27 & {\bf 102} & {\bf 370} & {\bf 1646} & {\bf 7437}
      \\
      \hline
    \end{tabular}
   \caption{The number of non equivalent Gauss diagrams of sizes  = 3, \ldots, 12, satisfying various     realizability conditions}
\label{tab:table1}
\end{table}

Table 1 shows the numbers of non-equivalent Gauss diagrams of sizes $3\ldots 12$ satisfying different realizability conditions, generated by our program.  The first two lines, found also in the OEIS as A264759 
together with the fact that generated corresponding lists of Gauss diagrams are the same,   verify that indeed, STZ correctly checks realizability up to the size 12.  

The last two lines show an interesting phenomenon.  Up to the size 8 B and GL conditions are satisfied by the same diagrams as  CA and STZ are.  For the size 9, however, there is one (= $102 - 101$)  up to equivalence Gauss diagram such that it satisfies both B and GL conditions, but is not realizable. For sizes 10,11 and 12 the numbers of such diagrams are 6, 36, and 235 respectively.

In summary, we can see that descriptions CA and STZ, on the one hand, and B and GL, on the other hand, are equivalent to each other up to the size 12. B and GL, however are not equivalent to CA or STZ, starting from size 9. Hence B and GL are not 
correct descriptions of
realizability, despite the claims in \cite{doi:10.1142/S0218216519500159} and \cite{grinblat2018realizability,doi:10.1142/S0218216520500315}, respectively. 

Employing our algorithms, we have also enumerated Gauss diagrams without imposing realizability conditions. Table 2 
demonstrates the numbers of non-equivalent prime Gauss diagrams (realizable and non-realizable) and the numbers of \emph{all} non-equivalent Gauss diagrams (not necessarily prime or realizable), which expands on the results of \cite{Valette16}.

\begin{table}[h]
 \centering
    \begin{tabular}{|l|c|c|c|c|c|c|c|c|c|l|} 
   \hline 
   \;\;\;\it{Size}  & 3 & 4 & 5 & 6 & 7 & 8 & 9 & 10 & 11 & Comments 
     \\
      \hline
      PRIME & 1 & 1 & 4 & 8 & 40 & 183 & 1354 & 11079 & 110026 & new  
      \\ 
      \hline 
      ALL & 3 & 5 & 17 & 53 & 260 & 1466 & 10915 & 93196 & $\ldots$ &  \cite{Valette16} (up to n=7)  \\
     \hline
    \end{tabular}
   \caption{The numbers of non equivalent prime and all Gauss diagrams}
    \label{tab:table1}
\end{table}

\subsection*{Counterexample, Size=9}

Fig~\ref{fig:min} represents the smallest counter-example; this is the only non-realizable Gauss diagram of size $9$ which satisfies both B and GL conditions.

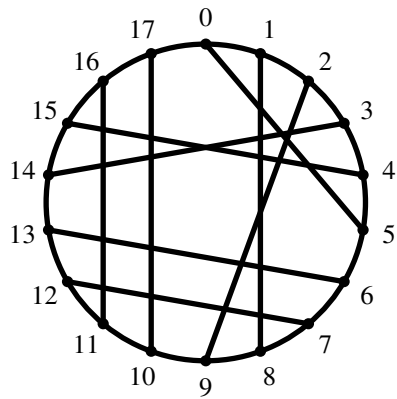
\begin{figure}[h!]
    \centering
     \begin{tikzpicture}[scale=0.7]
        \draw[line width =2] (0,0) circle (3);
     
     {\foreach \angle/ \label in
   { 90/0, 110/17, 130/16, 150/15, 170/14, 190/13, 210/12, 230/11, 250/10, 
     270/9, 290/8, 310/7, 330/6, 350/5, 10/4, 30/3, 50/2, 70/1 }
   {
    \fill(\angle:3.5) node{$\label$};
    \fill(\angle:3) circle (3pt) ;
    }
}

  \draw[line width = 2] (90:3) -- (350:3);
  \draw[line width = 2] (70:3) -- (290:3);
  \draw[line width = 2] (50:3) -- (270:3);
  \draw[line width = 2] (30:3) -- (170:3);
  \draw[line width = 2] (10:3) -- (150:3);
  \draw[line width = 2] (330:3) -- (190:3);
  \draw[line width = 2] (310:3) -- (210:3);
  \draw[line width = 2] (250:3) -- (110:3);
  \draw[line width = 2] (230:3) -- (130:3);

  \end{tikzpicture}
    \caption{Minimal contourexample, $n=9$}
    \label{fig:min}
\end{figure}

\subsection*{Counterexamples, Size=10}
Fig~\ref{fig:counter10} represents all $6$ counter-examples of size $10$; these are non-realizable Gauss diagram of size $10$ which satisfy both B and GL conditions.

\begin{figure}[h!]
 \begin{tikzpicture}[scale=0.8]
   \begin{scope}[line width = 2, scale=0.7]
\draw (0,0) circle (3);

 \draw[line width = 2] (90:3) -- (324:3);
 \draw[line width = 2] (162:3) -- (360:3);
 
 \draw[line width = 2] (18:3) -- (72:3);
 \draw[line width = 2] (126:3) -- (216:3);
 \draw[line width = 2] (108:3) -- (234:3);
 \draw[line width = 2] (198:3) -- (288:3);
 
 \draw[line width = 2] (144:3) -- (342:3);
 \draw[line width = 2] (180:3) -- (306:3);
 
 \draw[line width = 2] (252:3) -- (54:3);
 \draw[line width = 2] (36:3) -- (270:3);


  \end{scope}
  \begin{scope}[xshift = 6 cm,scale=0.7, line width = 2]
    \draw[line width =2] (0,0) circle (3);

 \draw[line width = 2] (162:3) -- (360:3);
 \draw[line width = 2] (90:3) -- (180:3);
 
 \draw (108:3) -- (234:3);
 \draw (126:3) -- (253:3);
 \draw (216:3) -- (306:3);
 \draw (216:3) -- (306:3);
 \draw (72:3) -- (18:3);
 
 \draw[line width = 2] (288:3) -- (54:3);
 \draw[line width = 2] (36:3) -- (270:3);

 \draw[line width = 2] (144:3) -- (342:3);
 \draw[line width = 2] (324:3) -- (198:3);

  \end{scope}
 
\begin{scope}[xshift = 12cm,scale = 0.7,line width = 2]
    \draw[line width =2] (0,0) circle (3);

 \draw[line width = 2] (324:3) -- (234:3);
 \draw[line width = 2] (72:3) -- (18:3);
 
 \draw (108:3) -- (198:3);
 \draw (162:3) -- (288:3);
 \draw (180:3) -- (306:3);
 \draw (36:3) -- (270:3);
 \draw (72:3) -- (18:3);
 
 \draw[line width = 2] (360:3) -- (126:3);
 \draw[line width = 2] (144:3) -- (342:3);
 
 \draw[line width = 2] (90:3) -- (216:3);
 \draw[line width = 2] (252:3) -- (54:3);

\end{scope}
\begin{scope}[yshift = -6 cm,scale = 0.7,line width = 2]
    \draw[line width =2] (0,0) circle (3);

 \draw[line width = 2] (252:3) -- (162:3);
 \draw[line width = 2] (306:3) -- (180:3);
 
 \draw (90:3) -- (216:3);
 \draw (108:3) -- (234:3);
 \draw (126:3) -- (360:3);
 \draw (36:3) -- (270:3);
 \draw (72:3) -- (18:3);
 
 \draw[line width = 2] (36:3) -- (270:3);
 \draw[line width = 2] (54:3) -- (288:3);
 
 \draw[line width = 2] (144:3) -- (342:3);
 \draw[line width = 2] (198:3) -- (324:3);

  \end{scope}
   \begin{scope}[yshift= - 6cm,xshift = 6cm, scale = 0.7,line width =2]
    \draw[line width =2] (0,0) circle (3);

 \draw[line width = 2] (144:3) -- (18:3);
 \draw[line width = 2] (270:3) -- (360:3);
 \draw[line width = 2] (72:3) -- (342:3);
 \draw[line width = 2] (324:3) -- (198:3);

 \draw (108:3) -- (234:3);
 \draw (126:3) -- (216:3);

 \draw[line width = 2] (252:3) -- (54:3);
 \draw[line width = 2] (288:3) -- (90:3);
 
 \draw[line width = 2] (162:3) -- (36:3);
 \draw[line width = 2] (306:3) -- (180:3);


  \end{scope}
  \begin{scope}[yshift= - 6cm,xshift = 12cm, scale = 0.7,line width =2]
    \draw[line width = 2] (0,0) circle (3);

 \draw[line width = 2] (108:3) -- (234:3);
 \draw[line width = 2] (144:3) -- (54:3);

 \draw (306:3) -- (180:3);
 \draw (198:3) -- (288:3);
 \draw (360:3) -- (270:3);
 
 \draw[line width = 2] (72:3) -- (342:3);
 \draw[line width = 2] (324:3) -- (90:3);
 
 \draw[line width = 2] (126:3) -- (216:3);
 \draw[line width = 2] (162:3) -- (36:3);
\draw[line width = 2] (252:3) -- (18:3);

  \end{scope}
 \end{tikzpicture}
 \caption{All counterexamples for $n=10$}
 \label{fig:counter10}
\end{figure}
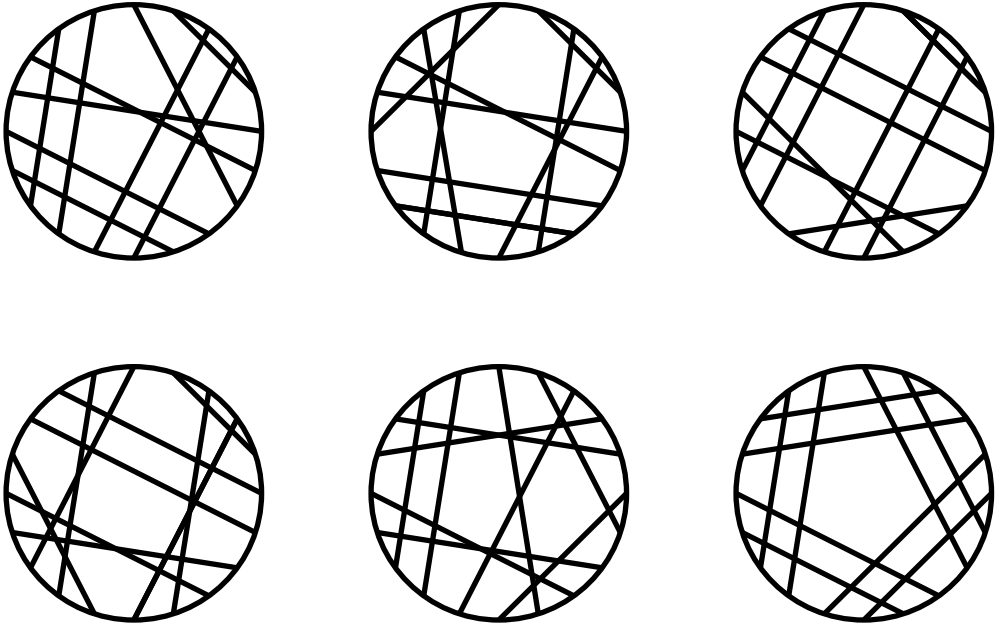


\section{New descriptions of realizability}

In this section, we aim to give a new description of the realizability of Gauss diagrams in terms of the adjacency matrix of its interlacement graphs by using STZ-condition. We show that the realizability of a Gauss diagram is equivalent to the existence of a solution of the corresponding system of linear equations over the field $GF(2).$ We finish the section with explaining why the \textbf{GL} and \textbf{B} conditions are incorrect.

\subsection{Reformulations of STZ conditions}

Let $M$ be the adjacency matrix of a Gauss diagram $G$. Since $M$ is symmetric, $M^2 = (\langle m_i, m_j \rangle)_{1\le i,j \le n}$, over $GF(2)$, where 
\[
\langle m_i, m_j\rangle: = m_{i,1}m_{j,1} + \cdots + m_{i,n}m_{j,n},
\]
and $m_k:=(m_{k,1}, \ldots, m_{k,n})$ is the $k$th row of the $M$.

Let $D$ be a diagonal $n\times n$ matrix, i.e., $D = \sum_{k \in K} E_{k,k}$, where $K \subseteq \{1,\ldots, n\}$, and $E_{k,k}$ the elementary matrix, that is $E_{k,k} = (e_{i,j})_{1\le i,j \le n}$ where $e_{i,j} = 1$ if and only if $i=j=k$ and $e_{i,j}=0$ otherwise. Since $M$ is symmetric then, if $K \ne \{1,\ldots,n\}$, $DM + MD = \sum_{k \in K} M_k$ where $M_k$ is obtained from $M$ by zeroing all elements of $M$ except elements of the $k$th row and the $k$th column.

Note that if $K = \{1,\ldots, n\}$ then $D$ is the identity $n\times n$ matrix, and the STZ-condition implies $M^2 = M$, i.e, $M$ is idempotent.

Thus the STZ-conditions can be reformulated as follows

\begin{proposition}
 Let $G$ be a Gauss diagram and $M$ the adjacency $n \times n$ matrix of its interlacement graph $I_G$. Then $G$ is realizable if and only if one of the following condtions hold:
 \begin{enumerate}
     \item whenever $m_{i,j} \equiv \alpha \bmod{2}$ then $\langle m_i, m_j \rangle \equiv \alpha \bmod{2}$, $1\le i,j \le n$, $\alpha = 0,1$
     \item there is a $K\subsetneq \{ 1,\ldots, n \}$ such that $M + M^2 = \sum_{k\in K}M_k.$ 
 \end{enumerate}  
\end{proposition}
\begin{proof}
 Indeed, the first condition holds when the matrix $M$ is idempotent, $M^2 = M$ and by the STZ-condition, $D$ is zero matrix. Next, if $M$ is not idempotent the STZ conditions imply that a diagonal matrix $D$ mast have a form $D = \sum_{k \in K}E_k$ where $K \subsetneq \{1,\ldots, n\}$ and the statement follows.
\end{proof}

 We thus can reformulate the STZ-conditions as follows
 \begin{theorem}\label{reformSTZ}
  Let $G$ be a Gauss diagram, $M = (m_{i,j})_{1\le i,j \le n}$ its adjacency matrix. Then the diagram $G$ is realizable if and only if the following system of equations
    \[
   \left\{ (\alpha_i + \alpha_j)m_{i,j} = \langle m_i, m_j \rangle + m_{i,j}, \quad 1\le i,j \le n  \right.
  \]
 has a solution over a field $GF(2).$ 
 \end{theorem}
 \begin{proof}
  Let $M$ be an adjacency matrix of a Gauss diagram $G$, say $M = (m_{i,j})_{1\le i,j \le n}$. Set $M':=M + M^2$, $M'=(m'_{i,j})_{1 \le i,j \le n}$. We have $m'_{i,j} = m_{i,j} + \langle m_i, m_j \rangle$, for any $1 \le i,j \le n.$

Next, since $M$ is symmetric with zero diagonal we then get
 \[
  \sum_{i=1}^n \alpha_i M_i = \begin{pmatrix}
   0 & (\alpha_1 + \alpha_2)m_{1,2} & (\alpha_1 + \alpha_3)m_{1,3} & \cdots & (\alpha_1 + \alpha_n)m_{1,n}\\
   (\alpha_1 + \alpha_2)m_{1,2} & 0 & (\alpha_2 + \alpha_3)m_{2,3} & \cdots & (\alpha_2 + \alpha_n)m_{2,n} \\
   \vdots & \vdots & \vdots & \ddots & \vdots & \\
   (\alpha_1 + \alpha_n)m_{1,n} & (\alpha_2 + \alpha_n)m_{2,n} & (\alpha_3 + \alpha_n)m_{3,n} & \cdots & 0
  \end{pmatrix}
 \]
 where $\alpha_1,\ldots, \alpha_n \in GF(2).$
 
 By the STZ-conditions, $M = M^2 + \sum_{k =1}^n \alpha_k M_k$, we thus get the following system of equations
 \[
 \{ (\alpha_i + \alpha_j)m_{i,j} = m'_{i,j}, \qquad 1 \le i,j \le n.
 \]
 
 By $m_{i,j}' = m_{i,j} + \langle m_i ,m_j \rangle$, the statement follows. It is clear that in the case $K = \{1,\ldots, n\}$ the statement holds because STZ-conditions implies that $M^2 = M$, i.e., all $\alpha_i = 1$. 
 \end{proof}

We now show that a Gauss diagram to be realisable it is enough to use STZ-conditions only instead of STZ-conditions plus \textbf{PC1}, \textbf{PC2}-conditions. 

 \begin{corollary}
  If a Gauss diagram does not satisfy \textbf{PC1} and \textbf{PC2} conditions then it is not realisable.
 \end{corollary}
 \begin{proof}
  Indeed, let $M = (m_{i,j})_{1 \le i,j \le n}$ be its adjacency matrix. By the assumptions there exist at least two $i,j$ such that $m_{i,j} = 0$ and $\langle m_i, m_j \rangle = 1$. Hence the system contains the equations $0=1$ that gives a contradiction. 
 \end{proof}
 
However, in practice to know whether a Gauss diagram is realizable it is useful to require conditions \textbf{PC1}, \textbf{PC2} are holding. In this case we have the following
 
\begin{corollary}\label{cor_of_system}
 Let a Gauss diagram $G$ satisfy \textbf{PC1} and \textbf{PC2} conditions, $M = (m_{i,j})_{1 \le i,j \le n}$ its adjacency matrix. Then $G$ is realisable if and only if the following system of equations
 \[
  \alpha_i + \alpha_j = \langle m_i, m_j \rangle + 1, \qquad i,j \in K
 \]
 where $K \subseteq \{1,\ldots, n\}$ is a subset such that whenever $i,j \in K$ then $m_{i,j} = 1$. 
\end{corollary}
\begin{proof}
 Indeed, if $G$ satisfies \textbf{PC1} and \textbf{PC2} conditions then whenever $m_{i,j} =0$ we have $\langle m_i, m_j \rangle = 0$ and by Proposition \ref{reformSTZ} the statement follows.
\end{proof}

\begin{corollary}\label{cor_for_graph}
 Let $I_G= (V,E)$ be an interlacement graph of a Guass diagram $G$, where $V$ is a set of vertices and $E \subseteq V \times V$ is a set of edges and let $M$ be its adjacency matrix. Consider the weighted graph $\widetilde{I_G}:=(V,E,\omega)$, where the edge weight function $\omega: E \to GF(2)$ is defined as follows $\omega: (i,j) \mapsto \langle m_i, m_j \rangle$. Then a Gauss diagram is realizable if and only if its interlacemnt graph is euler and for any cycle $C = (c_1,\ldots, c_\ell)$, we have 
 \[
 \sum_{i=1}^\ell \omega(c_i) \equiv \ell \bmod{2}.
 \]
\end{corollary}
\begin{proof}
 Indeed, by Corollary \ref{cor_of_system}, 
 \[
  \begin{cases}
   \alpha_{i_1} + \alpha_{i_2} \equiv \omega((i_1,i_2)) + 1, \\
   \phantom{\alpha_{i_1}\,}\vdots \phantom{+\alpha_{i_2}} \ddots \phantom{\omega((i_1,i_2))} \vdots \\
   \alpha_{i_\ell} + \alpha_{i_1} \equiv \omega((i_\ell, i_1)) + 1,
  \end{cases} 
 \]
 where we have put $c_1 = (i_1,i_2), \ldots, c_\ell = (i_\ell, i_1)$, and the statement thus follows.
\end{proof}

\begin{example}
 Let us consider the following Gauss diagram (see Fig.\ref{fig_for_STZ1}). 
 \begin{figure}[h!]
     \centering
     \begin{tikzpicture}[scale = 0.5]
      \draw[line width =2] (0,0) circle (3);
      {\foreach \angle/ \label in
       { 90/4, 120/3, 150/2, 180/1, 210/5, 240/6, 270/3, 300/4, 330/6, 
        0/2, 30/1, 60/5 
        }
     {
        \fill(\angle:3.5) node{$\label$};
        \fill(\angle:3) circle (3pt) ;
      }
    }
  \draw[line width = 2] (90:3) -- (300:3);
  \draw[line width = 2] (120:3) -- (270:3);
  \draw[line width = 2] (150:3) -- (0:3);
  \draw[line width = 2] (180:3) -- (30:3);
  \draw[line width = 2] (210:3) -- (60:3);
  \draw[line width = 2] (240:3) -- (330:3);
  
  \begin{scope}[xshift = 10cm]
     {\foreach \angle/ \label in
       { 90/1, 150/6,  210/5, 270/4,  330/3, 30/2 
        }
     {
        \fill(\angle:3.5) node{$\label$};
        \fill(\angle:3) circle (5pt) ;
      }
    }
     
  \draw[line width =2] (90:3) -- (30:3);     
  \draw (90:3) -- (210:3);
  \draw (90:3) -- (330:3);
  \draw (90:3) -- (270:3);
  \draw[line width =2] (90:3) -- (210:3);
  
  \draw (30:3) -- (330:3);
  \draw (30:3) -- (270:3);
  \draw[line width =2] (30:3) -- (210:3);
  
  \draw (330:3) -- (210:3);
  \draw (330:3) -- (150:3);
  
  \draw (270:3) -- (210:3);
  \draw (270:3) -- (150:3);

  \end{scope}
     \end{tikzpicture}
     \caption{The Gauss diagram $G$ and its weighted graph $I_G$; the thick edges have weight equal to $1$ and the other have weight equal to 0. We see that this graph satisfies the conditions of Corollary \ref{cor_for_graph} and thus the diagram is realizable}
     \label{fig_for_STZ1}
 \end{figure}
 
 We get
 \[
 M = \begin{pmatrix}
    0 & 1 & 1 & 1& 1& 0 \\
    1 & 0 & 1 & 1& 1& 0 \\
    1 & 1 & 0 & 0 & 1& 1 \\
    1 & 1 & 0 & 0 & 1& 1 \\
    1 & 1 & 1 & 1 & 0& 0 \\
    0 & 0 & 1 & 1 & 0& 0
 \end{pmatrix} \qquad
  M^2 = \begin{pmatrix}
    0 & 1 & 0 & 0 & 1 & 0 \\
    1 & 0 & 0 & 0 & 1 & 0 \\
    0 & 0 & 0 & 0 & 0 & 0 \\
    0 & 0 & 0 & 0 & 0 & 0 \\
    1 & 1 & 0 & 0 & 0 & 0 \\
    0 & 0 & 0 & 0 & 0 &  0
 \end{pmatrix} \qquad
 M':= M^2 + M = \begin{pmatrix}
    0 & 0 & 1 & 1 & 0 & 0 \\
    0 & 0 & 1 & 1 & 0 & 0 \\
    1 & 1 & 0 & 0 & 1 & 1 \\
    1 & 1 & 0 & 0 & 1 & 1 \\
    0 & 0 & 1 & 1 & 0 & 0 \\
    0 & 0 & 1 & 1 & 0 & 0
 \end{pmatrix}
 \]

We thus have the following system of equations
 \[
  \begin{cases}
   \alpha_1 + \alpha_2  = 0 \\
   \alpha_1 + \alpha_3  = 1 \\
   \alpha_1 + \alpha_4  = 1 \\
   \alpha_1 + \alpha_5  = 0 \\
   \alpha_2 + \alpha_3  = 1 \\
   \alpha_2 + \alpha_4  = 1 \\
   \alpha_2 + \alpha_5  = 0 \\
   \alpha_3 + \alpha_5  = 1 \\
   \alpha_3 + \alpha_6  = 1 \\
   \alpha_4 + \alpha_5  = 1 \\
   \alpha_4 + \alpha_6  = 1
  \end{cases}
 \]
 
 It follows that $\alpha_1 = c, \alpha_2 = c, \alpha_3 = 1+c, \alpha_4 = 1+ c, \alpha_5 = c, \alpha_6 = c$. Thus $D$ can be either $D = E_1 + E_2 + E_5 + E_6$ or $D = E_3 + E_4.$ Hence the corresponding Gauss diagram is realisable.
\end{example}

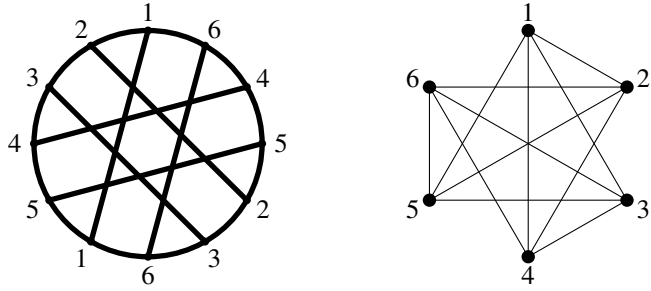
\begin{figure}[h!]
     \centering
     \begin{tikzpicture}[scale = 0.5]
      \draw[line width =2] (0,0) circle (3);
      {\foreach \angle/ \label in
       { 90/1, 120/2, 150/3, 180/4, 210/5, 240/1, 270/6, 300/3, 330/2, 
        0/5, 30/4, 60/6 
        }
     {
        \fill(\angle:3.5) node{$\label$};
        \fill(\angle:3) circle (3pt) ;
      }
    }
  \draw[line width = 2] (90:3) -- (240:3);
  \draw[line width = 2] (120:3) -- (330:3);
  \draw[line width = 2] (150:3) -- (300:3);
  \draw[line width = 2] (180:3) -- (30:3);
  \draw[line width = 2] (210:3) -- (0:3);
  \draw[line width = 2] (270:3) -- (60:3);
  
  \begin{scope}[xshift = 10cm]
     {\foreach \angle/ \label in
       { 90/1, 150/6,  210/5, 270/4,  330/3, 30/2 
        }
     {
        \fill(\angle:3.5) node{$\label$};
        \fill(\angle:3) circle (5pt) ;
      }
    }
     
  \draw (90:3) -- (30:3);     
  \draw (90:3) -- (330:3);
  \draw (90:3) -- (270:3);
  \draw (90:3) -- (210:3);

  \draw (30:3) -- (270:3);
  \draw (30:3) -- (210:3);
  \draw (30:3) -- (150:3);
  
  \draw (330:3) -- (270:3);
  \draw (330:3) -- (210:3);
  \draw (330:3) -- (150:3);
  
  \draw (270:3) -- (150:3);
  
  \draw (210:3) -- (150:3);

  \end{scope}
     \end{tikzpicture}
     \caption{Let us consider the following Gauss diagram $G$. It is easy to see that $M^2 = 0$ for its adjacency matrix $M$. Next, consider its weighted graph $I_G$; all its edges have thus weight equal to $0$. We see that this graph does not satisfy the conditions of Corollary \ref{cor_for_graph} because there exist cycles of length $3$ and thus the diagram is not realisable.}
     \label{fig_for_STZ1_not_realizable}
 \end{figure}

\subsection{A practical algorithm for testing if a graph is realizable} \label{sub:practical-algorithm-graph-realizable}

According to Theorem \ref{reformSTZ} a Gauss graph is realizable if and only if working in the two-element field $GF(2)$ we can choose values of parameters $\alpha_i$, $i=1,\dots,n$ in such a way that all equalities of the following form become true.

    \[
   \left\{ (\alpha_i + \alpha_j)m_{i,j} = \langle m_i, m_j \rangle + m_{i,j}, \quad 1\le i,j \le n  \right.
  \]

Recall that each parameter $\alpha_i$ can be treated as the color of the $i$-th vertex; the exact conditions which this $2$-colouring should satisfy will be revealed in the following paragraph.

Since we work in $GF(2)$, we can write an exhaustive list of all possible values of coefficients $m_{i,j}$ and $\langle m_i, m_j \rangle$. If $m_{i,j} = 0$ and $\langle m_i, m_j \rangle = 0$ then both the left-hand side and the right-hand side of the equality are $0$, so the equality is true. If $m_{i,j} = 0$ and $\langle m_i, m_j \rangle = 1$ then the left-hand side is $0$ and the right-hand side is $1$, therefore, the equality is false irrespective of the choice of $\alpha_i, \alpha_j$. In other words, if two vertices are not adjacent and they have an odd number of common neighbours, the graph is not realisable (this condition is the same as {\bf PC2}). If $m_{i,j} = 1$ and $\langle m_i, m_j \rangle = 0$ then the left-hand side is $\alpha_i + \alpha_j$ and the right-hand side is $1$. Therefore, we need to choose $\alpha_i, \alpha_j$ so that $\alpha_i \neq \alpha_j$. If $m_{i,j} = 1$ and $\langle m_i, m_j \rangle = 1$ then the left-hand side is $\alpha_i + \alpha_j$ and the right-hand side is $0$. Therefore, we need to choose $\alpha_i, \alpha_j$ so that $\alpha_i = \alpha_j$. 

Thus, the following algorithm can be used for checking if a Gauss graph is realisable. For convenience, we implement the conditions on the $2$-coloring from the description above as checking that a suitably modified graph is bipartite.

\begin{enumerate}
\item For each pair of vertices $u,v$ which are not adjacent to each other, count the number of their common neighbours. This number must be even for every such pair $u,v$.
\item Find all pairs of vertices $u,v$ such that $u,v$ are adjacent to each other and the number of their common neighbours is odd. After we have found all such pairs of vertices, for each such pair $u,v$ replace the edge $u, v$ by two new edges $u, w_{u,v}$ and $v, w_{u,v}$, where $w_{u,v}$ is a new vertex. Then check that the modified graph is bipartite. 
\end{enumerate}

\section{Comparison with planar Gauss diagrams}

Let us say that a Gauss diagram is \emph{planar} if all intersections of chords can be removed by laying out some of the chords on the outside of the circle. Planar Gauss diagrams have been employed in some steps of some algorithms for checking realizability of Gauss diagrams, starting from the first description of realizability, in which Dehn calls them `tree-and-onion diagrams' \cite{dehn1936}. Lemma 3 in \cite{DBLP:journals/jal/RosentiehlT84} is very close to our Proposition \ref{prop:planar-is-bipartite} below. However, we have never encountered a description of planarity of a Gauss diagram formulated on its own, and now is a good opportunity to do so, because the descriptions we give below look pleasantly similar to (but simpler than) Theorem \ref{reformSTZ} and the algorithm in Subsection \ref{sub:practical-algorithm-graph-realizable} above. 

\begin{lemma}
  Let $G$ be a Gauss diagram, $M = (m_{i,j})_{1\le i,j \le n}$ its adjacency matrix. For each non-zero entry $m_{i,j}$, consider an equation $\alpha_i + \alpha_j = 1$.
  The diagram $G$ is realizable if and only the system consisting of these equations
 has a solution over a field $GF(2).$ 
\end{lemma}
\begin{proof}
 Indeed, our aim is to lay out the chords inside and outside the circle to remove all intersections. We can conveniently encode this process by saying that $\alpha_i = 0$ [or $\alpha_i = 0$] means that the $i$-th chord is drawn inside [or outside] the circle. It is easy to see that solving the system of equations in the lemma is equivalent to laying out the chords in a way which removes all crossings. 
\end{proof}

\begin{proposition}
 \label{prop:planar-is-bipartite}
A Gauss diagram is planar if and only if its circle graph is bipartite. 
\end{proposition}
\begin{proof}
It is easy to see that a solution to the system of equations in the lemma exists if and only if the graph is bipartite. 
\end{proof}

\section{What was wrong in the GL description of realizability}

It is claimed in {\cite[Theorem 3.11]{doi:10.1142/S0218216520500315}} that 
 a Gauss diagram $G$ is realizable if and only if the following conditions hold:
  \begin{itemize}
    \item[(1)] the number of all chords that cross any two non-intersecting chords and the number of all chords intersecting each chord are both even (including zero),
    \item[(2)] for every chord $c \in G$ the Gauss diagram $\widehat{G}_c$ (= Conway's smoothing of the chord $c$) also satisfies the above condition.
  \end{itemize}

 It was shown (see proof of Theorem \cite[Theorem 4.3]{doi:10.1142/S0218216520500315}) that these conditions can be reformulated in term of adjacency matrix $M = (m_{i,j})_{1 \le i,j \le n}$ of $G$ as follows:
   \begin{enumerate}
      \item $\langle m_i, m_j \rangle \equiv \bmod{2}$, $1 \le i \le n$,
      \item $\langle m_i, m_j \rangle \equiv 0 \bmod{2}$, if the corresponding chords do not intersect,
      \item $\langle m_i, m_j \rangle + \langle m_i, m_k \rangle + \langle m_j, m_k \rangle \equiv 1 \bmod{2}$, if the corresponding chords intersect pairwise.
  \end{enumerate}

And we thus get a partial case of the STZ-conditions. Indeed, we get the system of equations (Corollary \ref{cor_of_system}) for three pairwise intersecting chords only. This is why the GL conditions are not complete. Finally, since \textbf{B3} conditions are mentioned in \cite[Proof of Theorem 4.3]{doi:10.1142/S0218216520500315} and it can be reformulated as (3) we then get the same reasons of its incompleteness.

Also in \cite{doi:10.1142/S0218216520500315} it was announced that if a Gauss diagram $G$ is realizable (say by a plane curve $\mathscr{C}(G)$), then to every closed path (say) $\mathscr{P}$ along $\mathscr{C}(G)$ we can associate a colouring of another part of $\mathscr{C}(G)$ into two colours (roughly speaking we get ``inner'' and ``outer'' sides of $\mathscr{P}$, \textit{cf.} Jordan curve Theorem). If a Gauss diagram is not realizable, then (\cite[Theorem 1.A]{gpv-ficvk-00}) it defines a virtual plane curve $\mathscr{C}(G)$. Next, it was shown that there exists a closed path along $\mathscr{C}(G)$ (a counter) to which we cannot associate a well-defined colouring of $\mathscr{C}(G)$, \textit{i.e.,} $\mathscr{C}(G)$ contains a path which is coloured into two colours.

Then in \cite{doi:10.1142/S0218216520500315} it was shown that in the case of non realizability of a Gauss diagram we always can find a special kind of such counters ($X$-counter), in terms of Gauss diagrams this counter is made by two intersecting chords. The counterexamples in Section \ref{sec:counterexamples} show that this statement fails. However the idea that if a Gauss diagram is not realizable then there is a counter to which we cannot associate a well-defined colouring is still true. We present the corresponding counters below.

\begin{figure}[h!]
 \begin{tikzpicture}[scale=0.8]
   \begin{scope}[scale=0.7]
    \draw[line width =5] (0,0) circle (3);
    \draw[line width =4,white] (0,0) circle (3);
     \begin{scope}[rotate =144]
      \draw[line width=4, gray] (3,0) arc(0:36:3);
     \end{scope}
     
     \begin{scope}[rotate =306]
      \draw[line width=4, gray] (3,0) arc(0:36:3);
     \end{scope}
     
     \begin{scope}[rotate =0]
      \draw[line width=1.5, black] (3,0) arc(0:90:3);
     \end{scope}
     

 \draw[line width = 2] (90:3) -- (324:3);
 \draw[line width = 2] (162:3) -- (360:3);
 
 \draw[line width = 1] (18:3) -- (72:3);
 \draw[line width = 1] (126:3) -- (216:3);
 \draw[line width = 1] (108:3) -- (234:3);
 \draw[line width = 1] (198:3) -- (288:3);
 
 \draw[line width = 3, gray] (144:3) -- (342:3);
 \draw[line width = 3, gray] (180:3) -- (306:3);
 
 \draw[line width = 2, dotted] (252:3) -- (54:3);
 \draw[line width = 2, dotted] (36:3) -- (270:3);

  
{\foreach \angle/ \label in
   { 90/9, 108/6, 126/1, 144/2, 162/3, 180/4, 198/5, 216/1, 234/6,252/7, 
    270/8, 288/5, 306/4, 324/9, 342/2, 360/3, 18/10, 36/8,   54/7, 72/10
   }
   {
    \fill(\angle:3.5) node{$\label$};
    \fill(\angle:3) circle (3pt) ;
    }
}
  \end{scope}
  \begin{scope}[xshift = 6 cm,scale=0.7]
    \draw[line width =5] (0,0) circle (3);
    \draw[line width =4,white] (0,0) circle (3);
     \begin{scope}[rotate =144]
      \draw[line width=4, gray] (3,0) arc(0:54:3);
     \end{scope}
     
     \begin{scope}[rotate =324]
      \draw[line width=4, gray] (3,0) arc(0:18:3);
     \end{scope}
     
     \begin{scope}[rotate =0]
      \draw[line width=1.5, black] (3,0) arc(0:90:3);
     \end{scope}
     

 \draw[line width = 2] (162:3) -- (360:3);
 \draw[line width = 2] (90:3) -- (180:3);
 
 \draw (108:3) -- (234:3);
 \draw (126:3) -- (253:3);
 \draw (216:3) -- (306:3);
 \draw (216:3) -- (306:3);
 \draw (72:3) -- (18:3);
 
 \draw[line width = 2, dotted] (288:3) -- (54:3);
 \draw[line width = 2, dotted] (36:3) -- (270:3);

 \draw[line width = 3, gray] (144:3) -- (342:3);
 \draw[line width = 3, gray] (324:3) -- (198:3);
  
{\foreach \angle/ \label in
   { 90/1, 108/2, 126/3, 144/4, 162/5, 180/1, 198/6, 216/7, 234/2,252/3, 
    270/8, 288/9, 306/7, 324/6, 342/4, 360/5, 18/10, 36/8,   54/9, 72/10
   }
   {
    \fill(\angle:3.5) node{$\label$};
    \fill(\angle:3) circle (3pt) ;
    }
}
  \end{scope}
 
\begin{scope}[xshift = 12cm,scale = 0.7]
    \draw[line width =5] (0,0) circle (3);
    \draw[line width =4,white] (0,0) circle (3);
     \begin{scope}[rotate = 216]
      \draw[line width=4, gray] (3,0) arc(0:36:3);
     \end{scope}
     
     \begin{scope}[rotate =54]
      \draw[line width=4, gray] (3,0) arc(0:36:3);
     \end{scope}
     
     \begin{scope}[rotate =324]
      \draw[line width=1.5, black] (3,0) arc(0:54:3);
     \end{scope}

 \draw[line width = 2] (324:3) -- (234:3);
 \draw[line width = 2] (72:3) -- (18:3);
 
 \draw (108:3) -- (198:3);
 \draw (162:3) -- (288:3);
 \draw (180:3) -- (306:3);
 \draw (36:3) -- (270:3);
 \draw (72:3) -- (18:3);
 
 \draw[line width = 2, dotted] (360:3) -- (126:3);
 \draw[line width = 2, dotted] (144:3) -- (342:3);
 
 \draw[line width = 3, gray] (90:3) -- (216:3);
 \draw[line width = 3, gray] (252:3) -- (54:3);
 {\foreach \angle/ \label in
   { 90/6, 108/1, 126/2, 144/3, 162/4, 180/5, 198/1, 216/6, 234/7,252/8, 
    270/9, 288/4, 306/5, 324/7, 342/3, 360/2, 18/10, 36/9,   54/8, 72/10
   }
   {
    \fill(\angle:3.5) node{$\label$};
    \fill(\angle:3) circle (3pt) ;
    }
}
\end{scope}
\begin{scope}[yshift = -6 cm,scale = 0.7]
    \draw[line width =5] (0,0) circle (3);
    \draw[line width =4,white] (0,0) circle (3);
     \begin{scope}[rotate = 144]
      \draw[line width=4, gray] (3,0) arc(0:54:3);
     \end{scope}
     
     \begin{scope}[rotate =324]
      \draw[line width=4, gray] (3,0) arc(0:18:3);
     \end{scope}
     
     \begin{scope}[rotate =252]
      \draw[line width=1.5, black] (3,0) arc(0:54:3);
     \end{scope}


 \draw[line width = 2] (252:3) -- (162:3);
 \draw[line width = 2] (306:3) -- (180:3);
 
 \draw (90:3) -- (216:3);
 \draw (108:3) -- (234:3);
 \draw (126:3) -- (360:3);
 \draw (36:3) -- (270:3);
 \draw (72:3) -- (18:3);
 
 \draw[line width = 2, dotted] (36:3) -- (270:3);
 \draw[line width = 2, dotted] (54:3) -- (288:3);
 
 \draw[line width = 3, gray] (144:3) -- (342:3);
 \draw[line width = 3, gray] (198:3) -- (324:3);
  
{\foreach \angle/ \label in
   { 90/1, 108/2, 126/3, 144/4, 162/5, 180/6, 198/7, 216/1, 234/2, 252/5, 
    270/8, 288/9, 306/6, 324/7, 342/4, 360/3, 18/10, 36/8,   54/9, 72/10
   }
   {
    \fill(\angle:3.5) node{$\label$};
    \fill(\angle:3) circle (3pt) ;
    }
}
  \end{scope}
   \begin{scope}[yshift= - 6cm,xshift = 6cm, scale = 0.7]
    \draw[line width =5] (0,0) circle (3);
    \draw[line width =4,white] (0,0) circle (3);
     \begin{scope}[rotate = 162]
      \draw[line width=4, gray] (3,0) arc(0:18:3);
     \end{scope}
     
     \begin{scope}[rotate =306]
      \draw[line width=4, gray] (3,0) arc(0:90:3);
     \end{scope}
     
     \begin{scope}[rotate =198]
      \draw[line width=1.5, black] (3,0) arc(0:72:3);
     \end{scope}
     
     \begin{scope}[rotate =72]
      \draw[line width=1.5, black] (3,0) arc(0:72:3);
     \end{scope}
 

 \draw[line width = 2] (144:3) -- (18:3);
 \draw[line width = 2] (270:3) -- (360:3);
 \draw[line width = 2] (72:3) -- (342:3);
 \draw[line width = 2] (324:3) -- (198:3);

 \draw (108:3) -- (234:3);
 \draw (126:3) -- (216:3);

 \draw[line width = 2, dotted] (252:3) -- (54:3);
 \draw[line width = 2, dotted] (288:3) -- (90:3);
 
 \draw[line width = 3, gray] (162:3) -- (36:3);
 \draw[line width = 3, gray] (306:3) -- (180:3);

  
{\foreach \angle/ \label in
   { 90/9, 108/1, 126/2, 144/3, 162/4, 180/5, 198/6, 216/2, 234/1, 252/7, 
    270/8, 288/9, 306/5, 324/6, 342/10, 360/8, 18/3, 36/4,   54/7, 72/10
   }
   {
    \fill(\angle:3.5) node{$\label$};
    \fill(\angle:3) circle (3pt) ;
    }
}
  \end{scope}
  \begin{scope}[yshift= - 6cm,xshift = 12cm, scale = 0.7]
    \draw[line width =5] (0,0) circle (3);
    \draw[line width =4,white] (0,0) circle (3);
     \begin{scope}[rotate = 126]
      \draw[line width=4, gray] (3,0) arc(0:36:3);
     \end{scope}
     
     \begin{scope}[rotate =216]
      \draw[line width=4, gray] (3,0) arc(0:36:3);
     \end{scope}
     
     \begin{scope}[rotate =18]
      \draw[line width=4, gray] (3,0) arc(0:18:3);
     \end{scope}

     \begin{scope}[rotate =54]
      \draw[line width=1.5, black] (3,0) arc(0:54:3);
     \end{scope}
 

 \draw[line width = 2] (108:3) -- (234:3);
 \draw[line width = 2] (144:3) -- (54:3);

 \draw (306:3) -- (180:3);
 \draw (198:3) -- (288:3);
 \draw (360:3) -- (270:3);
 
 \draw[line width = 2, dotted] (72:3) -- (342:3);
 \draw[line width = 2, dotted] (324:3) -- (90:3);
 
 \draw[line width = 3, gray] (126:3) -- (216:3);
 \draw[line width = 3, gray] (162:3) -- (36:3);
\draw[line width = 3, gray] (252:3) -- (18:3);
  
{\foreach \angle/ \label in
   { 90/9, 108/6, 126/1, 144/2, 162/3, 180/4, 198/5, 216/1, 234/6, 252/7, 
    270/8, 288/5, 306/4, 324/9, 342/10, 360/8, 18/7, 36/3,   54/2, 72/10
   }
   {
    \fill(\angle:3.5) node{$\label$};
    \fill(\angle:3) circle (3pt) ;
    }
}
  \end{scope}
 \end{tikzpicture}
 \caption{The contours and the corresponding coloring chords. Thus for every diagram we got a contour (the gray one) that does not divide it. Indeed, we have chords (dashed ones) with endpoints are in arcs with different colors. It follows that these diagrams are not realizable.}
 \label{Fig.8}
\end{figure}
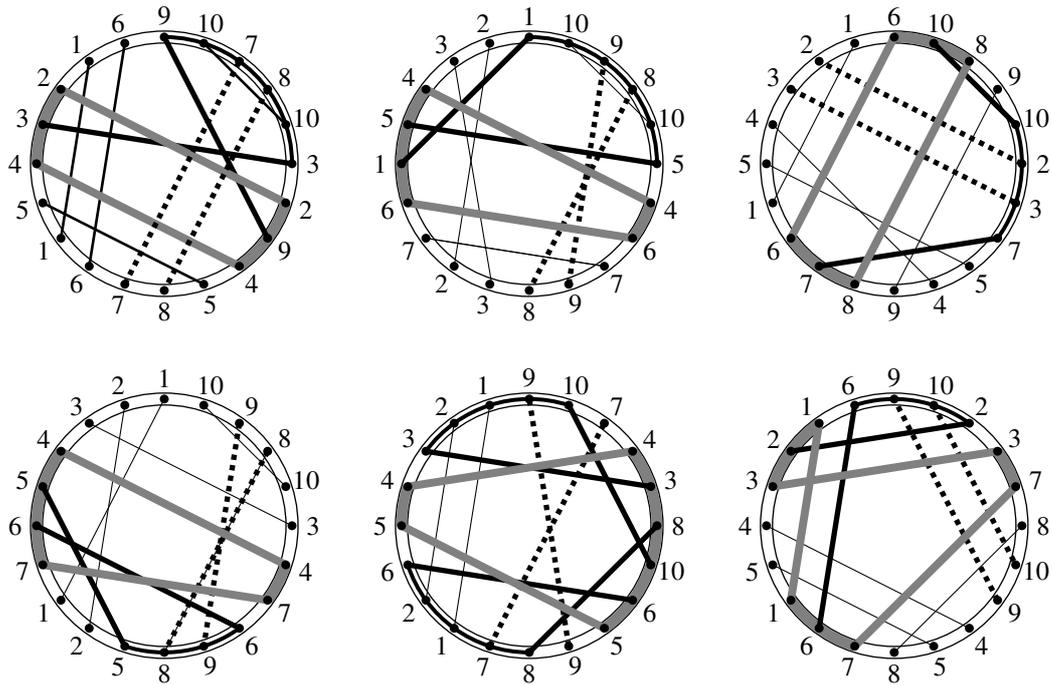

\section{Graphs corresponding to meanders}

In this section we consider the STZ-conditions when an adjacency matrix $M$ of a Gauss diagram is idempotent, i.e., $M^2= M$, which happens to be a special class of closed curves related to constructions known as meanders.

First of all we give an example of such matrices.

\begin{example}\label{meander1}
 Let us consider the following Gauss diagram
 
\begin{figure}[h!]
    \centering
    \begin{tikzpicture}[scale = 0.5]
            {\foreach \angle/ \label in
       {180/0, 155/1, 130/2,  105/3, 80/4, 55/5, 30/6, 0/0,
       330/1, 305/4, 280/3, 255/2, 230/5, 210/6
        }
     {
        \fill(\angle:3.5) node{$\label$};
        \fill(\angle:3) circle (4pt) ;
      }
    }
    
    \draw[line width = 2] (0,0) circle (3);
    
    \draw (0:3) -- (180:3);
    \draw (155:3) -- (330:3);
    \draw (130:3) -- (255:3);
    \draw (105:3) -- (280:3);
    \draw (80:3) -- (305:3);
    \draw (55:3) -- (230:3);
    \draw (30:3) -- (210:3);
    \end{tikzpicture}
\end{figure}
its adjacency matrix has the following form
 \[
  M = \begin{pmatrix} 0 & 1 & 1 & 1 & 1 & 1 & 1 \\
                      1 & 0 & 1 & 1 & 1 & 1 & 1 \\
                      1 & 1 & 0 & 0 & 0 & 1 & 1 \\
                      1 & 1 & 0 & 0 & 0 & 1 & 1 \\
                      1 & 1 & 0 & 0 & 0 & 1 & 1 \\
                      1 & 1 & 1 & 1 & 1 & 0 & 1 \\
                      1 & 1 & 1 & 1 & 1 & 1 & 0 
  \end{pmatrix}
 \]
 
By the straightforward verification it is easy to see that $M^2 = M$. The corresponding plane curve and graph $I_G$ are pictured below.

\begin{figure}[h!]
  \begin{center}
  \begin{tikzpicture}[scale=0.7]
  \begin{scope}[xshift=-3cm, line width = 2]
     \draw[name path =o] (0,0) to [out = 300, in = 200] (4,-0.5) to [out=20, in = 270] (5, 0.3) to [out = 90, in = 60] (0,0);
      \draw (0,0) to [out = 220, in = 200] (0.5,-1);
      \draw[name path = a] (0.5,-1) to [out = 20, in = 180] (3,0.5);
      \draw[name path =b] (3,0.5) to [out = 0, in = 30] (4.5,-1);
      \draw[name path =c] (4.5,-1) to [out = 210, in = 0] (2.7, -0.2);
      \draw[name path = d] (2.7, -0.2) to [out =  180, in = 90] (2,-1) to [out = 270, in = 270] (6,0) to [out = 90, in = 30] (4,2);
      \draw[name path =e] (4,2) to [out = 210, in = 0] (2,1);
      \draw[name path =f] (2,1) to [out = 180, in = 0] (0,1.5);
      \draw (0,1.5) to [out = 180, in = 120] (0,0);
     
     \fill [name intersections={of=o and a, by={A}}]
       (A) circle (4pt) node[below] {$1$};  
     
     \fill [name intersections={of=o and b, by={B}}]
       (B) circle (4pt) node[right] {$4$};
       
     \fill [name intersections={of=o and c, by={C}}]
       (C) circle (4pt) node[below] {$3$};
       
     \fill [name intersections={of=o and d, by={D}}]
       (D) circle (4pt) node[below left] {$2$}; 
      
     \fill [name intersections={of=o and e, by={E}}]
       (E) circle (4pt) node[above] {$5$};
       
     \fill [name intersections={of=o and f, by={F}}]
       (F) circle (4pt) node[above] {$6$}; 
     
       \fill (0,0) circle (4pt);
       \node[left] at (0,0){$0$};
    \end{scope}
    \begin{scope}[xshift = 7cm, scale = 0.7]
       {\foreach \angle/ \label in
       {180/0, 155/1, 130/2,  105/3, 80/4, 55/5, 30/6, 0/0,
       330/1, 305/4, 280/3, 255/2, 230/5, 210/6
        }
     {
        \fill(\angle:3.5) node{$\label$};
        \fill(\angle:3) circle (4pt) ;
      }
    }
    
    \draw[line width = 2] (0,0) circle (3);
    
    \draw (0:3) -- (180:3);
    \draw (155:3) -- (330:3);
    \draw (130:3) -- (255:3);
    \draw (105:3) -- (280:3);
    \draw (80:3) -- (305:3);
    \draw (55:3) -- (230:3);
    \draw (30:3) -- (210:3);
  \end{scope}
  \begin{scope}[xshift = 13cm,scale =0.7]
          {\foreach \angle/ \label in
       { 90/2, 30/3, 330/4, 270/5, 210/6, 150/1
        }
     {
        \fill(\angle:3.5) node{$\label$};
        \fill(\angle:3) circle (5pt) ;
        
        \draw (65:6) -- (\angle:3);
        \draw (150:3) -- (\angle:3);
              }
    }
    
    \draw (90:3) -- (270:3);
    \draw (90:3) -- (210:3);
    
    \draw (30:3) -- (270:3);
    \draw (30:3) -- (210:3);
    
    \draw (330:3) -- (270:3);
    \draw (330:3) -- (210:3);
    
    \draw (270:3) -- (210:3);
    
    \fill(65:6.5) node[left] {$0$};
    \fill(65:6) circle (5pt);
  \end{scope}
  
   \fill(-2,-3.3) node{$a)$};
  \fill(7,-3.3) node{$b)$};
  \fill(13,-3.3) node{$c)$}; 
\end{tikzpicture}
  \end{center}
  \caption{a) The plane curve, b) the Gauss diagram, c) the graph $I_G.$}\label{meander1}
\end{figure}
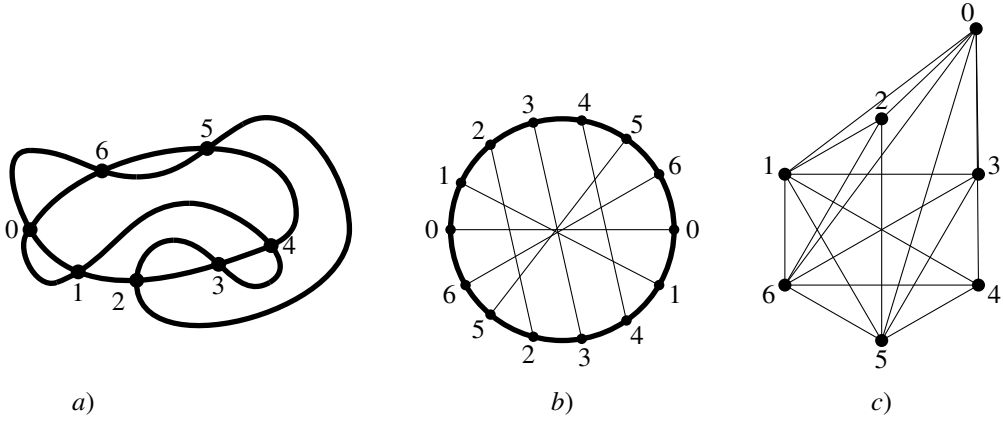
\end{example}

\begin{definition}
 An open meander is a configuration consisting of an oriented simple curve and a segment of a straight line on the plane that cross one another a finite number of times. Two open meanders are equivalent if there is a homeomorphism of the plane that maps one meander to the other.
\end{definition}

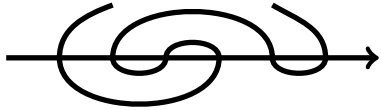
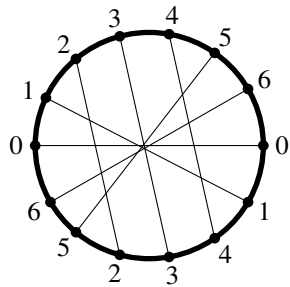
\begin{figure}[h!]
    \centering
     \begin{tikzpicture}[scale = 0.7]
      \draw[->, line width = 2] (0,0) to (7,0);
      \draw[line width =2] (2,1) to [out = 200, in = 90] (1,0);
      \draw[line width =2] (1,0) to [out = 270, in = 270] (4,0);
      \draw[line width =2] (4,0) to [out = 90, in = 90] (3,0);
      \draw[line width =2] (3,0) to [out = 270, in = 270] (2,0);
      \draw[line width =2] (2,0) to [out = 90, in = 90] (5,0);
      \draw[line width =2] (5,0) to [out = 270, in = 270] (6,0);
      \draw[line width =2] (6,0) to [out = 90, in = 330] (5,1);
     \end{tikzpicture}
    \caption{In other words, a meander can be also defines as follows. Take a fixed oriented line $L$ in $\mathbb{R}^2$, a meander of order $n$ (in this case $n = 6)$ is a non-self-intersecting curve in $\mathbb{R}^2$ which transversally intersects the line at $n$ points for some positive integer $n$.}
    \label{meander2}
\end{figure}

It is convenient to draw a meander of order $2n$ so that the straight line is horizontal and oriented from left to right, and the curve is oriented so that at the first (that is, leftmost) intersection of the curve and the straight line it is directed from top to bottom. The intersections allow a natural labelling by the integers $\{1, 2, . . . , 2n\}$, defined by their ordering along the line from the left to the right. These labels also have a cyclic order defined by their order along the closed curve.

\begin{definition}
 This cyclic order, when interpreted as a a full cycle in $\mathfrak{S}_{2n}$ (= the symmetric group), is referred to as a \textit{meandric permutation}. 
\end{definition}

It is clear that by the afore-mentioned conditions of making of meanders there is a one-to-one correspondence between meanders and meanders permutations.

\begin{construction}
 Let $\mathscr{M}$ be a meander with an oriented line $L$ and $\{1,2,\ldots, 2n\}$ the corresponding intersecting points of $\mathscr{M}$ with $L$. Let $L$ has a finite length. Take some points, say $a$, $b$ on $L$ such that $a$ is the beginning point of $L$ and not equal to 1, and $b$ is the end point of $L$ and not equal to $2n$. Next, we fix some orientation of $\mathscr{M}$ and take some points, say $m,m'$ on $M$ such that $m$ be the start point and $m'$ the finish point. 
 
 Further, let us attach two intervals $I =[i_0,i_1], J= [j_1,j_1]$ to the configuration as follows $i_0$ attach to $m$ and $i_1$ to $b$, $j_0$ to $m'$ and $j_1$ to $a$. In the result we get an intersection point of the attached intervals, denote it by $0$. We thus get a closed curve $\widetilde{\mathscr{M}}$.
\end{construction}

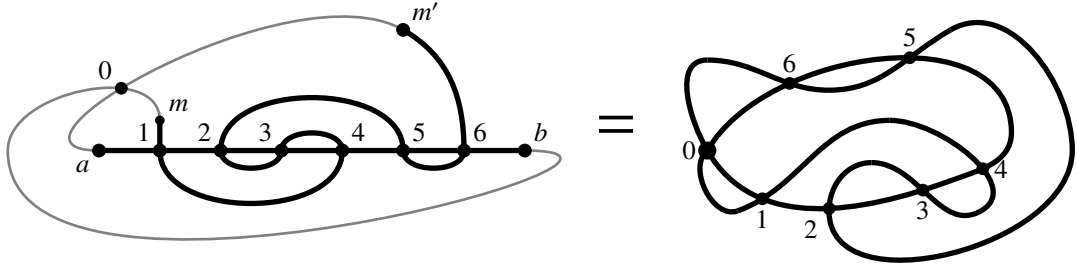
\begin{figure}
    \centering
     \begin{tikzpicture}[scale = 0.8]
      \begin{scope}[xshift = -3cm]
      \draw[line width = 2] (0,0) to (7,0);
      \draw[line width =2] (1,0.5) to [out = 270, in = 90] (1,0);
      \draw[line width =2] (1,0) to [out = 270, in = 270] (4,0);
      \draw[line width =2] (4,0) to [out = 90, in = 90] (3,0);
      \draw[line width =2] (3,0) to [out = 270, in = 270] (2,0);
      \draw[line width =2] (2,0) to [out = 90, in = 90] (5,0);
      \draw[line width =2] (5,0) to [out = 270, in = 270] (6,0);
      \draw[line width =2] (6,0) to [out = 90, in = 330] (5,2);
      
      \draw[line width =1,name path =a, gray] (1,0.5) to [out = 90, in = 90] (-1.5,0) to [out = 270, in = 0] (7,0);
      \draw[line width =1,name path =b,gray] (5,2) to [out = 150, in = 180] (0,0);
      
      \draw [fill](1,0) circle (3pt);
      \draw [fill](2,0) circle (3pt);
      \draw [fill](3,0) circle (3pt);
      \draw [fill](4,0) circle (3pt);
      \draw [fill](5,0) circle (3pt);
      \draw [fill](6,0) circle (3pt);
      
      \node[above left] at (1,0){$1$};
      \node[above left] at (2,0){$2$};
      \node[above left] at (3,0){$3$};
      \node[above right] at (4,0){$4$};
      \node[above right] at (5,0){$5$};
      \node[above right] at (6,0){$6$};
      
      \draw [fill](0,0) circle (3pt);
      \draw [fill](7,0) circle (3pt);
      \draw [fill](1,0.5) circle (2pt);
      \draw [fill](5,2) circle (3pt);
      
      \node[below left] at (0,0){$a$};
      \node[above right] at (7,0){$b$};
      
      \node[above right] at (1,0.5){$m$};
      \node[above right] at (5,2){$m'$};
      
          \fill [name intersections={of=a and b, by={A}}]
       (A) circle (3pt) node[above left] {$0$};   
      \end{scope}
      
      \node[above] at (5.5,0) {\Huge{$=$}};
       
       \begin{scope}[xshift = 7cm,line width = 2]
         \draw[name path =o] (0,0) to [out = 300, in = 200] (4,-0.5) to [out=20, in = 270] (5, 0.3) to [out = 90, in = 60] (0,0);
      \draw (0,0) to [out = 220, in = 200] (0.5,-1);
      \draw[name path = a] (0.5,-1) to [out = 20, in = 180] (3,0.5);
      \draw[name path =b] (3,0.5) to [out = 0, in = 30] (4.5,-1);
      \draw[name path =c] (4.5,-1) to [out = 210, in = 0] (2.7, -0.2);
      \draw[name path = d] (2.7, -0.2) to [out =  180, in = 90] (2,-1) to [out = 270, in = 270] (6,0) to [out = 90, in = 30] (4,2);
      \draw[name path =e] (4,2) to [out = 210, in = 0] (2,1);
      \draw[name path =f] (2,1) to [out = 180, in = 0] (0,1.5);
      \draw (0,1.5) to [out = 180, in = 120] (0,0);
     
     \fill [name intersections={of=o and a, by={A}}]
       (A) circle (3pt) node[below] {$1$};  
     
     \fill [name intersections={of=o and b, by={B}}]
       (B) circle (3pt) node[right] {$4$};
       
     \fill [name intersections={of=o and c, by={C}}]
       (C) circle (3pt) node[below] {$3$};
       
     \fill [name intersections={of=o and d, by={D}}]
       (D) circle (3pt) node[below left] {$2$}; 
      
     \fill [name intersections={of=o and e, by={E}}]
       (E) circle (3pt) node[above] {$5$};
       
     \fill [name intersections={of=o and f, by={F}}]
       (F) circle (3pt) node[above] {$6$}; 
     
       \draw [fill](0,0) circle (3pt);
       \node[left] at (0,0){$0$};
       \end{scope}
   \end{tikzpicture}
    \caption{We make a closed curve with a new intersection point $0$ and a Gauss diagram of this curve we define as a Gauss diagram of the meander.}
    \label{meander3}
\end{figure}

We see that such curves can be also described as plane curves that have a Gauss diagram to have at least one chord crosses all others. Recall that Example \ref{meander1} gives a curve with this property.

We also see that after adding the new crossing $0$ we meet labeled points $1,\ldots, 2n$ by walking on the curve (from top to bottom) in the reverse order than we meet them by walking on the same curve in the same direction but without adding the crossing $0$. 

For any permutation $\pi \in \mathfrak{S}_n$ we consider the following set of pairs $R_\pi \subseteq \{1,\ldots, n\} \times \{1,\ldots, n \}$,
\[
 R_\pi:=\{(i,j)\, | \, i < j \, \& \, \pi(i) > \pi(j)\}.
\]

\begin{remark}\label{R=braids}
 The set $R_\pi$ are called Thurston generators of the braid groups. They are braids with positive crossings and any two strands cross at least one. The elements of $R_\pi$ corresponds to crossing of strings, i.e., if $(i,j) \in R_\pi$ then $i$th and $j$th strands are crossed. 
\end{remark}

\begin{definition}
 Let $\mathscr{M}$ be a meander of an order $n$ and let $\pi$ the corresponding permutation. The meander graph $\Gamma(\mathscr{M}) = (V,E)$ consists of the set of vertices $V = \{0,1,2,\ldots, n\}$ and the set $E: = \bigcup_{1\le i \le n}\{(0,i)\} \cup \neg R_\pi$. 
\end{definition}

\begin{example}
 Let us turn out to the meander pictured in Fig.\ref{meander2}.
 
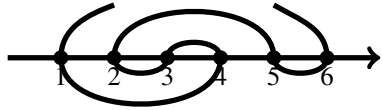
 \begin{figure}[h!]
    \centering
     \begin{tikzpicture}[scale = 0.7]
      \draw[->, line width = 2] (0,0) to (7,0);
      \draw[line width =2] (2,1) to [out = 200, in = 90] (1,0);
      \draw[line width =2] (1,0) to [out = 270, in = 270] (4,0);
      \draw[line width =2] (4,0) to [out = 90, in = 90] (3,0);
      \draw[line width =2] (3,0) to [out = 270, in = 270] (2,0);
      \draw[line width =2] (2,0) to [out = 90, in = 90] (5,0);
      \draw[line width =2] (5,0) to [out = 270, in = 270] (6,0);
      \draw[line width =2] (6,0) to [out = 90, in = 330] (5,1);
  
     {\foreach \x in
       { 1,2,3,4,5,6
        }
     {
        \fill(\x,0) node[below] {$\x$};
        \fill(\x,0) circle (4pt) ;
              }
    }
\end{tikzpicture}
    \label{meander4}
\end{figure}
 
 We have $n = 6$, $\pi = \begin{pmatrix} 1 & 2 & 3  & 4& 5& 6 \\ 1 & 4 & 3 & 2 & 5 & 6 \end{pmatrix}.$ Hence $R_\pi = \{ (2,3), (2,4), (3,4) \}$, therefore
 \[
  \neg R_{\pi} = \{ (1,2), (1,3), (1,4), (1,5), (1,6), (2,5), (2,6), (3,5), (3,6), (4,5), (4,6), (5,6) \}
 \]
 and we get exactly the same graph $I_G$ as in Fig.\ref{meander1} c).
 
\end{example}

\begin{lemma}\cite[Lemma 9.1.6]{EpThur}\label{criteria}
A set $R$ of pairs $(i,j)$, with $i <j$, comes from some permutation if and only if the following two conditions are satisfied:
\begin{enumerate}
    \item If $(i,j) \in R$ and $(j,k) \in R$, then $(i,k) \in R$.
    \item If $(i,k) \in R$, then $(i,j) \in R$ or $(j,k) \in R$ for every $j$ with $i<j<k$.
\end{enumerate}
\end{lemma}

 \begin{theorem}{(cf. \cite[Lemma 7.4, Theorem 7.5]{doi:10.1142/S0218216520500315})}\label{meander_matrix=idempotent}
  An adjacency matrix of any meander graph is idempotent.
\end{theorem}
\begin{proof}
 Indeed, we have $m_{0,i} = 1$ for any $1 \le i \le n$ where we have set that $M=(m_{i,j})_{0\le i,j \le n}$ is a matrix of a meander graph. Let $m_{i,j} = 1$ for some $1\le i,j \le n$ then the graph contains a cycle with three edges $(v_0,v_i), (v_i,v_j), (v_j,v_0)$ and by Corollary \ref{cor_for_graph}, $\langle m_0,m_i \rangle + \langle m_i, m_j \rangle + \langle m_0, m_j \rangle \equiv 1 \bmod{2}$. Since $\langle m_0, m_k \rangle  =1 + \langle m_k, m_k \rangle \equiv 1 \bmod{2}$ because of $\langle m_k, m_k \rangle \equiv 0 \bmod{2}$ for any $1 \le k \le n$. Therefore $\langle m_i, m_j \rangle \equiv 1 \bmod{2}$. We thus get $\langle m_i, m_j \rangle \equiv m_{i,j} \bmod{2}$ for any $0 \le i,j \le n$, i.e., $M^2 = M$ as claimed.
\end{proof}

\begin{theorem}
 Let $\Gamma = (V,E)$ be a finite graph with $V = \{0,1,\ldots, 2n\}$, $n >0$. The finite graph $\Gamma = (V,E)$ is meander graph if and only if the following conditions hold:
 \begin{enumerate}
     \item $(0,i) \in E$ for any $1\le i \le n$
     \item if $(i,j) \in E$ then either $(i,k) \in E$ or $(k,j) \in E$ for all $i < k < j$
     \item if $(i,j)\in E$ and $(j,k) \in E$ then $(i,k) \in E$ for any $0 \le i,j, k \le 2n$
     \item its adjacency matrix $M$ is idempotent.
 \end{enumerate}
\end{theorem}
\begin{proof}
 Since a meandric permutation defines an open meander up to homemomorphism it is enough to prove that such graphs defines a meandric permutation.
 
 Next, let us reformulate the conditions of the Theorem in term of adjacency matrices.
 
 Any matrix ${M} = ({m}_{ij}) \in \mathrm{Mat}_{n}(GF(2))$  is an adjacency matrix of a meander graph if and only if the following conditions hold:
  \begin{itemize}
    \item[(1)] its main diagonal contains only $0$,
    \item[(2)] ${M}$ is symmetric,
    \item[(3)] it has at least one string ${m}_i = ({m}_{i1}, \ldots, {m}_{i,n})$ with ${m}_{ik} = 1$ for every $1 \le k\le n$, $k \ne i$,
    \item[(4)] ${M}$ satisfies the STZ-conditions,
    \item[(5)] if ${m}_{pq}=1$ and ${m}_{qr} = 1$, then ${m}_{pr}=1$ for all $1 \le p,q,r \le n$,
    \item[(6)] if ${m}_{st}=1$, then either ${m}_{sj}=1$ or ${m}_{jt}=1$ for every $j$ with $1 \le s<j<t \le n$.
 \end{itemize}
 
 By the STZ-contions, \cite[Proposition 6.3]{doi:10.1142/S0218216520500315}, and Theorem \ref{meander_matrix=idempotent} the statement follows.
\end{proof}

\begin{remark}
 Just for convenience we omit the vertex $0$ in a meander graph.
\end{remark}

\begin{corollary}[{cf. meanders construction algorithm \cite[p. 22]{doi:10.1142/S0218216520500315}}]
  We get the following algorithm to construct meander graphs 
  \begin{itemize}
    \item[\sf{REQUIRE}] $N \equiv 0 \bmod{2}$, $V = V_0\cup V_1$, $V_0=\{2,4,\ldots, N\}$, $V_1 = \{1,3, \ldots, N-1\}$
    \item[\sf{ENSURE}] $(n_1,\ldots, n_N) \in V$
    \item[\bf{1.}] $i:=1$
    \item[{\bf{2.}}] Choose a vertex $n_i \in V_i$
    \item[\bf{3.}] Colour the $n_i$
    \item[\bf{4.}] Join the vertex $n$ with any vertex $n'$ if $n' >n$
    \item[\bf{5.}] \textsf{IF} the vertex $n$ has an odd number of all its neighbours \textsf{THEN GOTO} \textbf{6.} \textsf{ELSE END}
    \item[\bf{6.}] \textsf{IF} for any two coloured $n,m$ a number of its common neighbours is odd \textsf{THEN GOTO} \textbf{7.} \textsf{ELSE END}
    \item[\bf{7.}] \textsf{PRINT} $n_i$
    \item[\bf{8.}] $V_i:=V_i \setminus \{n_i\}$
    \item[\bf{9.}] $i:=i+1\bmod{2}$
    \item[\bf{10.}] \textsf{GOTO} \textbf{2.}
    \item[\bf{11.}] \textsf{END}
    \end{itemize}
\end{corollary}

\begin{example}
  Let $N=8$, then $V_0=\{2,4,6,8\}$ and $V_1=\{1,3,5,7\}$.
  \begin{enumerate}
      \item Picture them as it shown in Fig.\ref{meander5}.1).
      
      \item Let us choose vertex $3$ we then have to join it with vertices $4,5,6,7,8$ and colour it (we draw a circle around it instead of colouring) we get Fig.\ref{meander5}.2).
      
      \item It is clear that vertex $3$ has odd number of neighbours.
      We have to choose a vertex with even number, say, $6$. We then have to colour it and joint with $7,8$, we obtain Fig.\ref{meander5}.3).
      
      \item  We see that common numbers of neighbours for $6$ is odd. We have to choose a vertex with an odd number, say $7$. We then get Fig.\ref{meander5}.4).
      
      \item Again, the oddness conditions hold. We have to choose a vertex with an even number. Let us choose a vertex $4$ we then get Fig.\ref{meander5}.5).
      
      \item  We see that, for instance, vertex $6$ has an even number of neighbours it follows that we cannot choose vertex $4$ on this step. Let us choose vertex $2$, we then have to join it with vertices $4,5,8$, we get Fig.\ref{meander5}.6). It is easy to verify that the oddness conditions hold.
      
      \item We have to chose a vertex with an odd number. It is easy to see that we have only one choice -- vertex $1.$ Indeed, if we chose vertex $5$ we get Fig.\ref{meander5}.7). We see that vertices $5$, $6$ have an even number of common neighbours; $3,8$, i.e., we cannot chose vertex $5$ on this step.
      
      \item Choose vertex $1$ and we thus get Fig.\ref{meander5}.8).
      
      \item We have to chose a vertex with an even number. If we chose vertex $8$ we get Fig.\ref{meander5}.9). We see that vertices $4,6$ have an even number of common neighbours; $3,8$, i.e., we cannon chose a vertex $4$ on this step and thus we chose vertex $8$. We obtain Fig.\ref{meander5}.10).
      
      \item Thus we must choose $5$ and $4$ and we then get the following meandric permutation
         \[
           \pi = \begin{pmatrix} 1 &2&3&4&5&6&7&8\\3&6&7&2&1&8&5&4 \end{pmatrix}
         \]
and the corresponding meander has the following form as it shown if Fig.\ref{meander6}.
  \end{enumerate}

 \begin{figure}[h!]
     \centering
     \begin{tikzpicture}[scale = 0.41]
    \begin{scope}[xshift = -10cm]
     {\foreach \angle/ \label in
       { 90/1, 45/2, 0/3, 315/4, 270/5, 225/6, 180/7, 135/8
        }
     {
        \fill(\angle:3.5) node{$\label$};
        \fill(\angle:3) circle (5pt) ;
      }
    } 
    \node[below] at (270:5){1)};
    \end{scope}    
    \begin{scope}[xshift=0cm]
     {\foreach \angle/ \label in
       { 90/1, 45/2, 0/3, 315/4, 270/5, 225/6, 180/7, 135/8
        }
     {
        \fill(\angle:3.5) node{$\label$};
        \fill(\angle:3) circle (5pt) ;
      }
    }
    \draw (0:3) -- (315:3);
    \draw (0:3) -- (270:3);
    \draw (0:3) -- (225:3);
    \draw (0:3) -- (180:3);
    \draw (0:3) -- (135:3);
    
    \draw(0:3) circle(10pt);
    
    \node[below] at (270:5){2)};
    \end{scope}
   \begin{scope}[xshift = 10cm]
      {\foreach \angle/ \label in
       { 90/1, 45/2, 0/3, 315/4, 270/5, 225/6, 180/7, 135/8
        }
     {
        \fill(\angle:3.5) node{$\label$};
        \fill(\angle:3) circle (5pt) ;
      }
    }
    \draw (0:3) -- (315:3);
    \draw (0:3) -- (270:3);
    \draw (0:3) -- (225:3);
    \draw (0:3) -- (180:3);
    \draw (0:3) -- (135:3);
    
    \draw (225:3) -- (180:3);
    \draw (225:3) -- (135:3);

    \draw(0:3) circle(10pt);
    \draw(225:3) circle (10pt);
     \node[below] at (270:5){3)};
   \end{scope}
   \begin{scope}[xshift = -10cm, yshift = -12 cm]
      {\foreach \angle/ \label in
       { 90/1, 45/2, 0/3, 315/4, 270/5, 225/6, 180/7, 135/8
        }
     {
        \fill(\angle:3.5) node{$\label$};
        \fill(\angle:3) circle (5pt) ;
      }
    }
    \draw (0:3) -- (315:3);
    \draw (0:3) -- (270:3);
    \draw (0:3) -- (225:3);
    \draw (0:3) -- (180:3);
    \draw (0:3) -- (135:3);
    
    \draw (225:3) -- (180:3);
    \draw (225:3) -- (135:3);
    
    \draw(180:3) -- (135:3);
    
    \draw(0:3) circle(10pt);
    \draw(225:3) circle (10pt);
    \draw(180:3) circle (10pt);
       \node[below] at (270:5){4)};
   \end{scope}
   \begin{scope}[yshift = -12cm]
      {\foreach \angle/ \label in
       { 90/1, 45/2, 0/3, 315/4, 270/5, 225/6, 180/7, 135/8
        }
     {
        \fill(\angle:3.5) node{$\label$};
        \fill(\angle:3) circle (5pt) ;
      }
    }
    \draw (0:3) -- (315:3);
    \draw (0:3) -- (270:3);
    \draw (0:3) -- (225:3);
    \draw (0:3) -- (180:3);
    \draw (0:3) -- (135:3);
    
    \draw (225:3) -- (180:3);
    \draw (225:3) -- (135:3);
    
    \draw(180:3) -- (135:3);
    
    \draw(315:3) -- (270:3);
    \draw(315:3) -- (225:3);
    \draw(315:3) -- (180:3);
    \draw(315:3) -- (135:3);

    \draw(0:3) circle(10pt);
    \draw(225:3) circle (10pt);
    \draw(180:3) circle (10pt);
    \draw(315:3) circle (10pt);
    
   \node[below] at (270:5){5)};
   \end{scope}
   \begin{scope}[xshift = 10cm, yshift = -12cm]
       {\foreach \angle/ \label in
       { 90/1, 45/2, 0/3, 315/4, 270/5, 225/6, 180/7, 135/8
        }
     {
        \fill(\angle:3.5) node{$\label$};
        \fill(\angle:3) circle (5pt) ;
      }
    }
    \draw (0:3) -- (315:3);
    \draw (0:3) -- (270:3);
    \draw (0:3) -- (225:3);
    \draw (0:3) -- (180:3);
    \draw (0:3) -- (135:3);
    
    \draw (225:3) -- (180:3);
    \draw (225:3) -- (135:3);
    
    \draw(180:3) -- (135:3);
    
    \draw(45:3) -- (315:3);
    \draw(45:3) -- (270:3);
    \draw(45:3) -- (135:3);

    \draw(0:3) circle(10pt);
    \draw(225:3) circle (10pt);
    \draw(180:3) circle (10pt);
    \draw(45:3) circle (10pt);
    
 \node[below] at (270:5){6)};
   \end{scope}
 \begin{scope}[yshift = -24cm, xshift = - 10cm]
     {\foreach \angle/ \label in
       { 90/1, 45/2, 0/3, 315/4, 270/5, 225/6, 180/7, 135/8
        }
     {
        \fill(\angle:3.5) node{$\label$};
        \fill(\angle:3) circle (5pt) ;
      }
    }
    \draw (0:3) -- (315:3);
    \draw (0:3) -- (270:3);
    \draw (0:3) -- (225:3);
    \draw (0:3) -- (180:3);
    \draw (0:3) -- (135:3);
    
    \draw (225:3) -- (180:3);
    \draw (225:3) -- (135:3);
    
    \draw(180:3) -- (135:3);
    
    \draw(45:3) -- (315:3);
    \draw(45:3) -- (270:3);
    \draw(45:3) -- (135:3);
    
    \draw(270:3) -- (135:3);

    \draw(0:3) circle(10pt);
    \draw(225:3) circle (10pt);
    \draw(180:3) circle (10pt);
    \draw(45:3) circle (10pt);
    \draw(270:3) circle (10pt);
    
  \node[below] at (270:5){7)};     
 \end{scope}
  \begin{scope}[yshift = -24cm]
     {\foreach \angle/ \label in
       { 90/1, 45/2, 0/3, 315/4, 270/5, 225/6, 180/7, 135/8
        }
     {
        \fill(\angle:3.5) node{$\label$};
        \fill(\angle:3) circle (5pt) ;
      }
    }
    \draw (0:3) -- (315:3);
    \draw (0:3) -- (270:3);
    \draw (0:3) -- (225:3);
    \draw (0:3) -- (180:3);
    \draw (0:3) -- (135:3);
    
    \draw (225:3) -- (180:3);
    \draw (225:3) -- (135:3);
    
    \draw(180:3) -- (135:3);
    
    \draw(45:3) -- (315:3);
    \draw(45:3) -- (270:3);
    \draw(45:3) -- (135:3);
    
    \draw(90:3) --(315:3);
    \draw(90:3) --(270:3);
    \draw(90:3) --(135:3);

    \draw(0:3) circle(10pt);
    \draw(225:3) circle (10pt);
    \draw(180:3) circle (10pt);
    \draw(45:3) circle (10pt);
    \draw(90:3) circle (10pt);
     \node[below] at (270:5){8)};  
  \end{scope}
  \begin{scope}[yshift = - 24cm, xshift = 10cm]
      {\foreach \angle/ \label in
       { 90/1, 45/2, 0/3, 315/4, 270/5, 225/6, 180/7, 135/8
        }
     {
        \fill(\angle:3.5) node{$\label$};
        \fill(\angle:3) circle (5pt) ;
      }
    }
    \draw (0:3) -- (315:3);
    \draw (0:3) -- (270:3);
    \draw (0:3) -- (225:3);
    \draw (0:3) -- (180:3);
    \draw (0:3) -- (135:3);
    
    \draw (225:3) -- (180:3);
    \draw (225:3) -- (135:3);
    
    \draw(180:3) -- (135:3);
    
    \draw(45:3) -- (315:3);
    \draw(45:3) -- (270:3);
    \draw(45:3) -- (135:3);
    
    \draw(90:3) --(315:3);
    \draw(90:3) --(270:3);
    \draw(90:3) --(135:3);
    
    \draw(315:3) --(270:3);
    \draw(315:3) --(135:3);

    \draw(0:3) circle(10pt);
    \draw(225:3) circle (10pt);
    \draw(180:3) circle (10pt);
    \draw(45:3) circle (10pt);
    \draw(90:3) circle (10pt);
    \draw(315:3) circle (10pt);
    
      \node[below] at (270:5){9)};  
  \end{scope}
 \begin{scope}[yshift = -36cm]
      {\foreach \angle/ \label in
       { 90/1, 45/2, 0/3, 315/4, 270/5, 225/6, 180/7, 135/8
        }
     {
        \fill(\angle:3.5) node{$\label$};
        \fill(\angle:3) circle (5pt) ;
      }
    }
    \draw (0:3) -- (315:3);
    \draw (0:3) -- (270:3);
    \draw (0:3) -- (225:3);
    \draw (0:3) -- (180:3);
    \draw (0:3) -- (135:3);
    
    \draw (225:3) -- (180:3);
    \draw (225:3) -- (135:3);
    
    \draw(180:3) -- (135:3);
    
    \draw(45:3) -- (315:3);
    \draw(45:3) -- (270:3);
    \draw(45:3) -- (135:3);
    
    \draw(90:3) --(315:3);
    \draw(90:3) --(270:3);
    \draw(90:3) --(135:3);

    \draw(0:3) circle(10pt);
    \draw(225:3) circle (10pt);
    \draw(180:3) circle (10pt);
    \draw(45:3) circle (10pt);
    \draw(90:3) circle (10pt);
    \draw(135:3) circle (10pt);
    
    \node[below] at (270:5){10)};
 \end{scope}
 \end{tikzpicture}
 \caption{Steps of making of a meander graph.}\label{meander5}
 \end{figure}
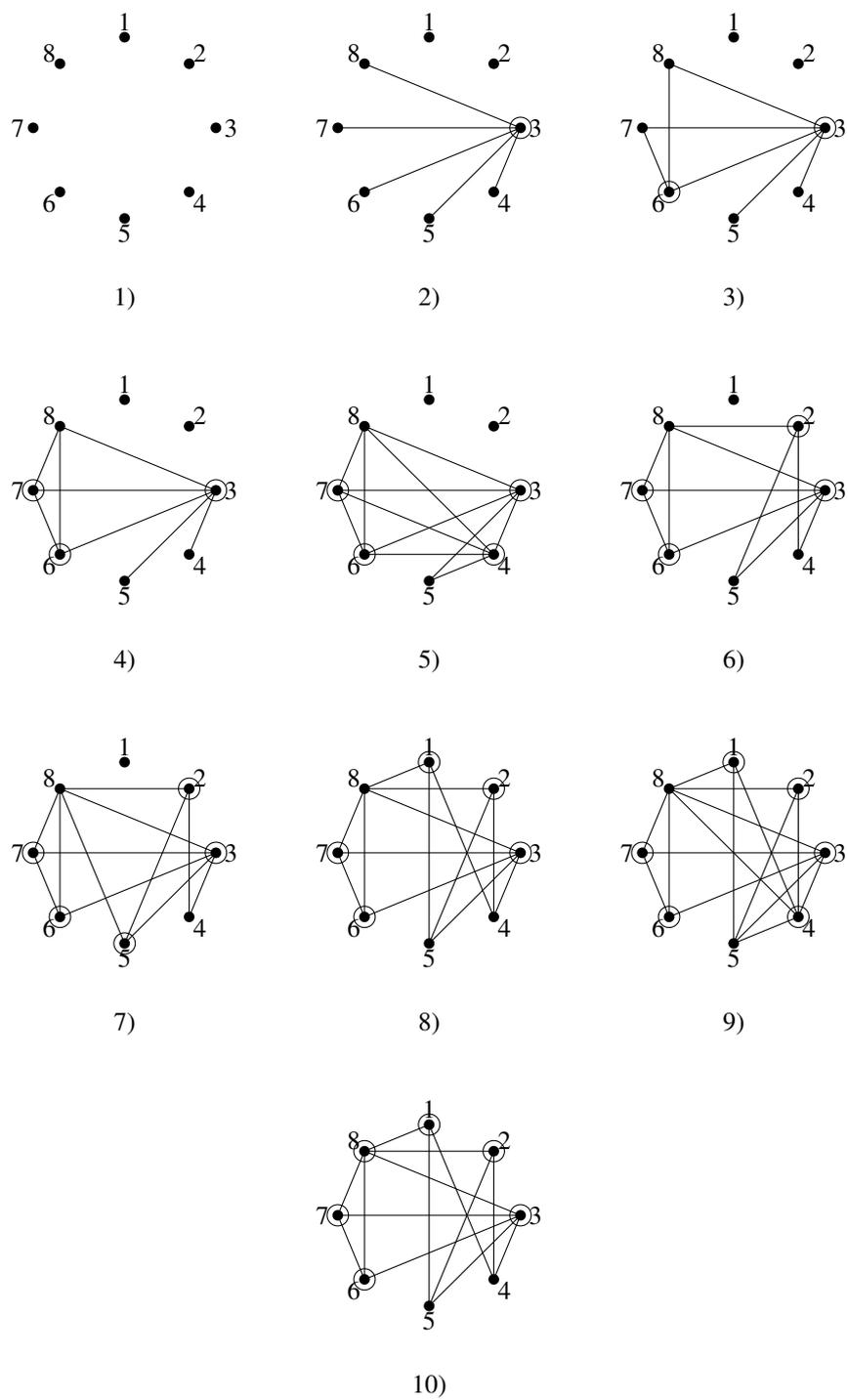

 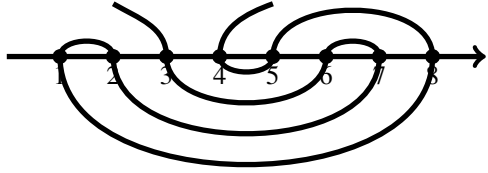
\begin{figure}[h!]
    \centering
     \begin{tikzpicture}[scale = 0.7]
      \draw[->, line width = 2] (0,0) to (9,0);
      \draw[line width =2] (2,1) to [out = 330, in = 90] (3,0);
      \draw[line width =2] (3,0) to [out = 270, in = 270] (6,0);
      \draw[line width =2] (6,0) to [out = 90, in = 90] (7,0);
      \draw[line width =2] (7,0) to [out = 270, in = 270] (2,0);
      \draw[line width =2] (2,0) to [out = 90, in = 90] (1,0);
      \draw[line width =2] (1,0) to [out = 270, in = 270] (8,0);
      \draw[line width =2] (8,0) to [out = 90, in = 90] (5,0);
      \draw[line width =2] (5,0) to [out = 270, in = 270] (4,0);
      \draw[line width =2] (4,0) to [out = 90, in = 200] (5,1);
  
     {\foreach \x in
       { 1,2,3,4,5,6,7,8
        }
     {
        \fill(\x,0) node[below] {$\x$};
        \fill(\x,0) circle (4pt) ;
              }
    }
\end{tikzpicture}
   \caption{The meander corresponds to the permutation $\pi$} \label{meander6}
\end{figure}

\end{example}

\section{Counting Gauss graphs}

By a Gauss graph we mean a circle graph of a realizable Gauss diagram. Recall from the previous section that a meander diagram is a Gauss diagram in which there is a chord intersecting all other chords and, accordingly, a meander graph is a Gauss graph in which there is a vertex adjacent to every other vertex. Table \ref{tab:table3} presents the number of realizable Gauss diagrams and realizable meander diagrams. In Table \ref{tab:table4} we have calculated the numbers of non-isomorphic realizable Gauss graphs (we uploaded it to the OEIS as A343358) and the numbers of non-isomorphic realizable meander graphs for sizes $n=1 \ldots 13$ (we uploaded it to the OEIS as A338660). 


%
 %


\begin{table}[h]
 \centering
    \begin{tabular}{|l|c|c|c|c|c|c|c|c|c|c|c|} 
   \hline 
   \;\;\;\it{Size}  & 3 & 4 & 5 & 6 & 7 & 8 & 9 & 10 & 11 & 12 & 13
     \\
      \hline
     Gauss diagrams, A264759  & 1 & 1 & 2 & 3 & 10 & 27 & 101 & 364 & 1610 & 7202 &  34659
      \\ 
      \hline 
      Meander diagrams & 1 &  & 2 & & 6 &  & 23 &  & 115 & & 688  \\
     \hline
    \end{tabular}
   \caption{The number of realizable Gauss diagrams  and meander  diagrams   of sizes  = 3, \ldots, 13}
    \label{tab:table3}
\end{table}

\begin{table}[h]
 \centering

    \begin{tabular}{|l|c|c|c|c|c|c|c|c|c|c|c|} 
   \hline 
   \;\;\;\it{Size}  & 3 & 4 & 5 & 6 & 7 & 8 & 9 & 10 & 11 & 12 & 13
     \\
      \hline
     Gauss graphs, A343358  & 1 & 1 & 2 & 3 & 7 & 18 & 41 & 123 & 361 & 1257 & 4573
      \\ 
      \hline 
      Meander graphs, A338660 & 1 &  & 2 & & 5 &  & 13 &  & 43 & & 167  \\
     \hline
    \end{tabular}
   \caption{The number of realizable  Gauss graphs and meander graphs  of sizes  = 3, \ldots, 13}
    \label{tab:table4}
    \end{table}

The calculated sequence A343358 coincides with OEIS sequence  
A002864 (Number of alternating prime knots with n crossings) up to n=10.  For  n=11, 12, 13  the numbers of alternating knots are different from those of Gauss graphs and are equal to  367, 1288, 4878 respectively. This observation can be explained as follows. Firstly, so called \emph{flype} move applied to a chord diagram does not change a graph of the diagram 
\cite{Soulie2004,chmutov_duzhin_mostovoy_2012}. 
According to \cite{Flype91},  who proved long-standing  Tait conjecture (from 1898),  any two reduced diagrams of an alternating  prime knot are related by a sequence of flype moves. Further to that,  the graphs of the diagrams are preserved by the \emph{mutation} moves, so non-equivalent mutant knots also have isomorphic Gauss graphs of their diagrams \cite{chmutov_duzhin_mostovoy_2012}.   
Thus, the number of non-isomorphic Gauss graphs should 
be \emph{no more}  than 
the number of alternating knots modulo equivalence generated by 
mutation 
moves (mutant knots).  It is well-known that no alternating mutant knots exist up to the size 10 
-- that explains coincidence of A343358 and A002864 up to n=10. Stoimenov's Knot Data Tables\footnote{http://stoimenov.net/stoimeno/homepage/ptab/}    present explicit lists of mutant knots (both alternating and non alternating)  for n = 11\ldots 15. For example, for n=11 there exist 6 pairs of mutant alternating knots (e.g. those related by flype moves) and that explains the case n=11: $367 = 361 + 6$.  For n=12 Knot Data Table lists 27 pairs of alternating mutant knots and 2 \emph{triples} of alternating mutant knots, which again consistent with our results, for $2188 = 1257 + 27 + 2 \times 2$. For size $13$, this manual comparison becomes difficult, but we wrote a script which counts mutant alternating knots in Stoimenov's knot data; it is known that there are $4878$ alternating knots of size $13$, and Stoimenov's list of mutant knots of size $13$ includes $574$ alternating knots in $269$ groups of mutant alternating knots, that is, there are $305$ more alternating knots than their groups, and $4878-305$ is our entry $4573$.
Thus, we not only proposed a new natural sequence (A343358) for OEIS, but contributed for cross validation of both theoretical results and empirically computed data in knot theory. 

An important theoretical question is whether the number of circle graphs of realizable Gauss diagrams of a given size is always equal to the number of classes of mutant alternating knots of this size, as in examples above, or whether these numbers may diverge for larger values of size. An important theorem \cite[Theorem in Section 4.8.5]{chmutov_duzhin_mostovoy_2012}  states that the circle graphs of two chord diagrams are isomorphic if and only if the two chord diagrams can be transformed into one another by a sequence of moves which are equivalent to knot mutations. Therefore, A343358 is both the sequence of numbers of circle graphs of realizable Gauss diagrams and the sequence of numbers of classes of mutant alternating knots.

\subsection{Remarks on implementation}

The numbers of non-isomorphic Gauss graphs have been computed by 
1) computing all non-equivalent Gauss diagrams using our implementations of the 
permutation-based algorithm
(up to n=11) and the incremental algorithm (up to n=12), for the case n=13 we have used Tait Curves program by J. Betrema\footnote{  https://github.com/j2b2/TaitCurves}; 
2) translating the results to Gauss codes; 3) applying  an implementation of circle graphs isomorphism procedure by calling Networkx library from Python code.  

\section{Conclusion}
The experimental approach we have applied to studying the  realizability of Gauss diagrams and related objects proved to be fruitful. With the help  of  two novel algorithms for efficient generation of Gauss diagrams we have found the errors in the recently published realizability criteria and proposed corrections. Based on that we further proposed new realizability criteria. We experimentally enumerated various classes of the diagrams and their interlacement graphs cross-validating and expanding on existing knowledge in the field.      

\section{Acknowledgments}
The work of the first, second and fourth named authors was supported by the Leverhulme Trust Research Project Grant RPG-2019-313.

\bibliographystyle{alpha}
\bibliography{sample}

\begin{thebibliography}{dFOdM97}

\bibitem[Bir19]{doi:10.1142/S0218216519500159}
Oleg~N. Biryukov.
\newblock Parity conditions for realizability of {Gauss} diagrams.
\newblock {\em Journal of Knot Theory and Its Ramifications}, 28(01):1950015,
  2019.

\bibitem[Car91]{c-cic-91}
J.~Scott Carter.
\newblock Classifying immersed curves.
\newblock {\em Proc. Amer. Math. Soc.}, 111(1):281--287, 1991.

\bibitem[CDM12]{chmutov_duzhin_mostovoy_2012}
S.~Chmutov, S.~Duzhin, and J.~Mostovoy.
\newblock {\em Introduction to Vassiliev Knot Invariants}.
\newblock Cambridge University Press, 2012.

\bibitem[CE96]{ce-pp2-96}
Grant Cairns and Daniel~M. Elton.
\newblock The planarity porblem {II}.
\newblock {\em J. Knot Theory Ramif.}, 5(2):137--144, 1996.

\bibitem[C.F00]{Gauss}
C.F.Gauss.
\newblock Werke, 1900.

\bibitem[CHLR06]{chmutov-2006}
Michael Chmutov, Thomas Hulse, Andrew Lum, and Peter Rowell.
\newblock {Plane and spherical curves: an investigation of their invariants}.
\newblock In {\em {Research Experiences for Undergraduates (REU), Summer
  Mathematics Research Institute, REU 2006 Proceedings}}, Oregon State
  University, 2006.

\bibitem[CW94]{cw-prpmg-94}
Nathalie Chaves and Claude Weber.
\newblock Plombages de rubans et probl\`eme des mots de {Gauss}.
\newblock {\em Expo. Math.}, 12:53--77, 1994.

\bibitem[{de }81]{DEFRAYSSEIX198129}
Hubert {de Fraysseix}.
\newblock Local complementation and interlacement graphs.
\newblock {\em Discrete Mathematics}, 33(1):29 -- 35, 1981.

\bibitem[Deh36]{dehn1936}
M.~Dehn.
\newblock Uber kombinatorische topologie.
\newblock {\em Acta Math.}, 67:123--168, 1936.

\bibitem[dFOdM97]{10.1007/3-540-63938-1_65}
H.~de~Fraysseix and P.~Ossona~de Mendez.
\newblock A short proof of a {Gauss} problem.
\newblock In Giuseppe DiBattista, editor, {\em Graph Drawing}, pages 230--235,
  Berlin, Heidelberg, 1997. Springer Berlin Heidelberg.

\bibitem[DT83]{DOWKER198319}
C.H. Dowker and Morwen~B. Thistlethwaite.
\newblock Classification of knot projections.
\newblock {\em Topology and its Applications}, 16(1):19 -- 31, 1983.

\bibitem[Fra69]{FRANCIS1969331}
George~K. Francis.
\newblock Null genus realizability criterion for abstract intersection
  sequences.
\newblock {\em Journal of Combinatorial Theory}, 7(4):331 -- 341, 1969.

\bibitem[GL18]{grinblat2018realizability}
Andrey Grinblat and Viktor Lopatkin.
\newblock On realizability of {Gauss} diagrams and constructions of meanders,
  2018.
\newblock arxiv:1808.08542.

\bibitem[GL20]{doi:10.1142/S0218216520500315}
Andrey Grinblat and Viktor Lopatkin.
\newblock On realizabilty of {Gauss} diagrams and constructions of meanders.
\newblock {\em Journal of Knot Theory and Its Ramifications}, 29(05):2050031,
  2020.

\bibitem[GPV00]{gpv-ficvk-00}
Mikhail Goussarov, Michael Polyak, and Oleg Viro.
\newblock Finite-type invariants of classical and virtual knots.
\newblock {\em Topology}, 39(5):1045--1068, 2000.

\bibitem[Kau99]{KAUFFMAN1999663}
Louis~H. Kauffman.
\newblock Virtual knot theory.
\newblock {\em European Journal of Combinatorics}, 20(7):663 -- 691, 1999.

\bibitem[KLV21a]{khan2021experimental}
A.~Khan, A.~Lisitsa, and A.~Vernitski.
\newblock Experimental mathematics approach to {Gauss} diagrams realizability.
\newblock arxiv:2103.02102, 2021.

\bibitem[KLV21b]{KLV21-lintel}
Abdullah Khan, Alexei Lisitsa, and Alexei Vernitski.
\newblock Gauss-lintel, an algorithm suite for exploring chord diagrams.
\newblock In Fairouz Kamareddine and Claudio Sacerdoti~Coen, editors, {\em
  Intelligent Computer Mathematics}, pages 197--202, Cham, 2021. Springer
  International Publishing.

\bibitem[LM76]{lovasz1976}
L\'aszl\'o Lov\'asz and Morris~L. Marx.
\newblock A forbidden substructure characterization of {Gauss} codes.
\newblock {\em Bull. Amer. Math. Soc.}, 82(1):121--122, 01 1976.

\bibitem[LPS11]{10.1007/978-3-642-21254-3_29}
Alexei Lisitsa, Igor Potapov, and Rafiq Saleh.
\newblock Planarity of knots, register automata and logspace computability.
\newblock In Adrian-Horia Dediu, Shunsuke Inenaga, and Carlos Mart{\'i}n-Vide,
  editors, {\em Language and Automata Theory and Applications}, pages 366--377,
  Berlin, Heidelberg, 2011. Springer Berlin Heidelberg.

\bibitem[Lyn54]{Lyndon}
R.~C. Lyndon.
\newblock {On Burnside's problem}.
\newblock {\em Trans. Amer. Math. Soc.}, 77:202--215, 1954.

\bibitem[Mar69]{10.2307/2037443}
Morris~L. Marx.
\newblock The {Gauss} realizability problem.
\newblock {\em Proceedings of the American Mathematical Society},
  22(3):610--613, 1969.

\bibitem[MT91]{Flype91}
William~W. Menasco and Morwen~B. Thistlethwaite.
\newblock The {Tait} flyping conjecture.
\newblock {\em Bull. Amer. Math. Soc.}, 25:403--412, 1991.

\bibitem[Ros76]{rosenstiehl:hal-00259712}
Pierre Rosenstiehl.
\newblock {Solution alg{\'e}brique du probl{\`e}me de Gauss sur la permutation
  des points d'intersection d'une ou plusieurs courbes ferm{\'e}es du plan}.
\newblock {\em {C.R. Acad. Sci.}}, t. 283, s{\'e}rie A:551--553, 1976.

\bibitem[RR76]{rosenstiehl:hal-00259721}
Pierre Rosenstiehl and R.C. Read.
\newblock {On the Gauss crossing problem}.
\newblock In A.Hajnal and V.T.~Sos eds., editors, {\em {On the Gauss crossing
  problem}}, volume~II, pages 843--876, Hungary, June 1976.

\bibitem[RT84]{DBLP:journals/jal/RosentiehlT84}
Pierre Rosenstiehl and Robert~Endre Tarjan.
\newblock Gauss codes, planar hamiltonian graphs, and stack-sortable
  permutations.
\newblock {\em J. Algorithms}, 5(3):375--390, 1984.

\bibitem[Sou04]{Soulie2004}
Christian Souli\'e.
\newblock Complete invariant graphs of alternating knots.
\newblock https://arxiv.org/ftp/math/papers/0404/0404490.pdf, April 2004.
\newblock arXiv:math/0404490.

\bibitem[STZ09]{DBLP:journals/dm/ShtyllaTZ09}
Blerta Shtylla, Lorenzo Traldi, and Louis Zulli.
\newblock On the realization of double occurrence words.
\newblock {\em Discret. Math.}, 309(6):1769--1773, 2009.

\bibitem[Val16]{Valette16}
G.~Valette.
\newblock {A classification of spherical curves based on Gauss diagrams}.
\newblock {\em Arnold Math J.}, 2:383--405, 2016.

\bibitem[Ven18]{vena2018topological}
Lluis Vena.
\newblock A topological characterization of {Gauss} codes, 2018.
\newblock arXiv:1808.04630.

\bibitem[WSTL12]{wielemaker:2011:tplp}
Jan Wielemaker, Tom Schrijvers, Markus Triska, and Torbj\"o{}rn Lager.
\newblock {SWI-Prolog}.
\newblock {\em Theory and Practice of Logic Programming}, 12(1-2):67--96, 2012.

\end{thebibliography}

\end{document}